\documentclass[a4paper]{amsart}
\usepackage[all]{xy}

\def\AA{{\mathbb{A}}}

\def\RR{{\mathbb{R}}}
\def\CC{{\mathbb{C}}}
\def\QQ{{\mathbb{Q}}}
\def\NN{{\mathbb{N}}}
\def\ZZ{{\mathbb{Z}}}

\let \cedilla =\c

\renewcommand{\b}{{\frak{b}}}

\renewcommand{\a}{{\frak{a}}}

\renewcommand{\o}{{\mathcal O}}
\newcommand{\mld}{{\mathrm{mld}}}

\newcommand{\cont}{{\mathrm{Cont}}}
\newcommand{\codim}{{\mathrm{codim}}}
\newcommand{\ord}{{\mathrm{ord}}}

\newcommand{\spec}{{\mathrm{Spec \ }}}

\newcommand{\val}{{\mathrm{val}}}

\newcommand{\m}{{\frak{m}}}
\newcommand{\tm}{{\widetilde\m}}

\newcommand{\height}{{\mathrm{ht} }}

\newcommand{\lct}{{\operatorname{lct}}}
\renewcommand{\th}{{\tilde{h}}}

\renewcommand{\wr}{{\widetilde R}}
\newcommand{\tf}{{\tilde f}}
\newcommand{\we}{{\widetilde E}}

\newcommand{\ta}{{\widetilde\a}}

\newcommand{\wq}{{\widetilde{Q}}}

\newcommand{\zp}{{\ZZ/(p)}}

\renewcommand{\sp}{{S\otimes_\ZZ \ZZ/(p)}}
\renewcommand{\mod}{{\operatorname{mod}\ }}

\newcommand{\of}{{\overline{f}}}

\newtheorem{thm}{Theorem}[section]
\newtheorem{cor}[thm]{Corollary}
\newtheorem{Corollary-Definition}[thm]{Corollary-Definition}
\newtheorem{prop-def}[thm]{Proposition-Definition}
\newtheorem{prop}[thm]{Proposition}
\newtheorem{lem}[thm]{Lemma}
\newtheorem{prob}[thm]{Problem}

\newtheorem{conj}[thm]{Conjecture}

\theoremstyle{definition}
\newtheorem{defn}[thm]{Definition}
\newtheorem{exmp}[thm]{Example}
\newtheorem{rem}[thm]{Remark}

\newtheorem{Proposition-Definition}[thm]{Proposition-Definition}

\begin{document}

\title[Inversion of modulo $p$ reduction]{Inversion of ``modulo $p$ reduction" \\
and a partial descent from characteristic 0 to positive characteristic  }

\thanks{
The author is partially supported by Grant-In-Aid (c) 1605089 of JSPS in Japan.}

\author{Shihoko Ishii }
\maketitle
\begin{abstract}

In this paper we focus on pairs consisting of the affine $N$-space and
multiideals  with a positive exponent.
We introduce a method ``lifting to characteristic $0$" which is a kind of the inversion of 
``modulo $p$ reduction".
By making use of it,
we prove that  Musta\cedilla{t}\v{a}-Nakamura's conjecture  and some 
uniform bound of divisors computing log canonical thresholds descend
from characteristic 0 to   certain classes of pairs in positive characteristic.
We also pose a problem whose affirmative answer gives the descent of the statements
to the whole set of pairs in positive characteristic.

\end{abstract}
\section{Introduction}
\noindent
For studies of singularities in characteristic 0, there are many tools; resolutions of 
the singularities,  Bertini's theorem (generic smoothness), many kinds of vanishing theorems, etc.,
which are  not available  for singularities in positive characteristic.
So, in order to avoid these difficulties, one  way is to reduce our problems in positive characteristic  into the problem in characteristic 0.
In this paper, we introduce ``lifting to characteristic $0$"  which is a kind of inversion of 
``modulo $p$ reduction" 
and show that some statements in characteristic $0$ descend 
into a certain class of pairs
 in positive 
characteristic.

      In this paper we focus 
      on a special object that is
       a pair $(A, \a^e)$ consisting of the  affine space $A=\AA_k^N$
       over an algebraically closed field $k$ 
     and a ``multiideal" $\a^e=\a_1^{e_1}\cdots\a_s^{e_s}$ on $A$  with the exponent 
     $e=\{e_1,\ldots,e_s\}\subset \RR_{>0}$.
     When we say ``a pair", we always mean this object.
     Sometimes we treat the case $s=1, e_1=1$.
     
\begin{defn}
   Let $k$ be an algebraically closed field of characteristic $p>0$ and
   $\mathcal S_k$ a class of pairs $(\AA_k^N, \a^e)$.
   Let $\mathcal T_\CC$ be a class of pairs $(\AA_\CC^N, \ta^e)$.
   We say that a statement $P$ in $\mathcal T_\CC$ {\sl descends to} $\mathcal S_k$,
   when the following holds:
   
   \begin{center}
   If $P$ holds in $\mathcal T_\CC$,  then $P$ also holds in $\mathcal S_k$.
\end{center}
\end{defn}     
Note that we do not discuss about that $P$ actually holds in $\mathcal T_\CC$
but discuss about the possibility for $P$ in $\mathcal T_\CC$ to descend to $\mathcal S_k$.
In this paper we consider Musta\cedilla{t}\v{a}-Nakamura's conjecture, ACC conjecture
  for minimal
log discrepancies and a uniform bound of divisors computing log canonical thresholds 
as the statement $P$.

Musta\cedilla{t}\v{a}-Nakamura's conjecture is as follows:

\begin{conj}[$M_{N, e}$]
     Let $A=\AA_k^N$ be  defined over 
     an  algebraically closed field $k$ and let $0\in A$ be the origin. 
     Given a finite subset 
     $e\subset \RR_{>0}$, there
     is a positive integer $\ell_{N,e}$ (depending on $N$ and $e$) such that for every multiideal 
     $\a^e$ on $A$ with the
    exponent  $e$, there is a prime divisor $E$ that computes $\mld(0; A,  \a^e)$ and 
    satisfies
    $k_E \leq \ell_{N,e}$.
\end{conj}
This conjecture is very useful, for example it guarantees lower semi continuity of the map
$$\AA_k^N\to \RR_{\geq0}\cup\{-\infty\};  \ \ x\mapsto \mld(x;\AA_k^N, \a^e),$$
and also stability of log canonicity under deformations,
which are not known in positive characteristic. 

Ascending chain condition (ACC) conjecture, which is important in birational geometry,
 is as follows:
\begin{conj}[$A_N$]\label{acc}
   Let $A$, $N$ and $0$   be as above. 
  For every fixed DCC set $J \subset \RR_{>0}$, the set
$$\{\mld(0; A, \a^e) \mid  e\subset J, (A, \a^e)\ \mbox{is\ log\ canonical\ at\ }
      0\}  $$
  satisfies ascending chain condition (ACC).     
\end{conj}
   Note that, in characteristic $0$, $M_{N,e}$ holds for any finite exponent $e$ if and only if $A_N$ holds 
 by \cite{mn} and \cite{kawk2}.

Another statement we would like to let descend from characteristic $0$ to positive
characteristic is the following which is proved by Shibata in characteristic $0$
 (for the proof, see  the appendix of this paper).
\begin{prop}\label{shibata2}
    Let $\m$ be the maximal ideal defining the origin $0\in \AA_\CC^N$.
    Then, for  a positive integer $\mu$,
    there is a positive integer $L_{N,\mu}$ depending only on $N$ and $\mu$ such that for every 
    ${\m}$-primary ideal $\a$ on $\AA_\CC^N$ with 
       ${\mathfrak m}^\mu\subset \a$, 
    there is some prime divisor $E$ over $\AA_\CC^N$ that computes $\mathrm{lct}(0;\AA_\CC^N,\a)$  
    and satisfies  $k_E\le L_{N,\mu}$.
\end{prop}

Idealistically, these three statements  would descend
from a set  of pairs $(\AA_\CC^N, \ta^e)$ to the whole set of pairs $(\AA_k^N,\a^e)$
for a field $k$ of positive characteristic.
However, what we obtain in this paper are for some restricted set of pairs $(\AA_k^N,\a^e)$.
Now we give the restriction for pairs.

\begin{defn} Let $k$ be an algebraically closed field of arbitrary characteristic.
  For a pair $(\AA_k^N,\a^e)$, we say that the minimal log discrepancy 
  $\mld(0; \AA_k^N, \a^e)$ {\sl has a toric approximation} if
  either
  \begin{enumerate}
    \item[] for every $\epsilon>0$, there exists a toric prime divisor $E_\epsilon$ over $\AA_k^N$
    such that 
    $$| a(E_\epsilon; \AA_k^N,\a^e)-\mld(0; \AA_k^N,\a^e) | <\epsilon,$$
    \item[]or  there exists a toric prime divisor over $\AA_k^N$ computing $\mld(0; \AA_k^N, \a^e)$.

  \end{enumerate}
  Here, a toric prime divisor means a prime divisor over $\AA_k^N$ such that the corresponding
  valuation is a discrete monomial valuation (see  Definition \ref{toric} for details).
\end{defn}

\begin{defn}  Let $k$ be an algebraically closed field of arbitrary characteristic.
  Let $\a$ be a coherent ideal sheaf on $\AA_k^N$.
  For a pair $(\AA_k^N,\a)$, we say that the log canonical threshold
  $\lct (0; \AA_k^N,\a) $ {\sl has a toric approximation} if 
    for every $\epsilon>0$
    there exists  a toric prime divisor $E_\epsilon$ over $\AA_k^N$
    such that 
   $$\left| \frac{k_{E_\epsilon}+1}{\val_{E_\epsilon}\a} - \lct(0; \AA_k^N, \a)\right| <\epsilon.$$

\end{defn}

\begin{defn} For an algebraically closed field $k$, 
   we define the sets $\mathcal M_k$ and $\mathcal L_k$  of pairs as follows:
\begin{enumerate}
\item[]
   $\mathcal M_k=\{(\AA_k^N, \a^e) \mid N\geq 1, e\subset \RR_{>0},
   \mld(0; \AA_k^N, \a^e)\ \mbox{has\ a\ toric\ approximation}\},$
\item[]   
   $\mathcal L_k=\{(\AA_k^N, \a) \mid N\geq 1,
   \lct(0; \AA_k^N, \a^e)\ \mbox{ has\ a \ toric\ approximation}\}.$
\end{enumerate} 
\end{defn}

The results of this paper are the following:

\begin{thm}\label{mainthm} Let $k$ be an algebraically closed field of characteristic $p>0$.
  Musta\cedilla{t}\v{a}-Nakamura's conjecture  for all pairs in characteristic
  $0$   
  descend to $\mathcal M_k$.
  And in this case there is a  prime divisor over $\AA_k^N$ 
  computing $\mld(0; \AA_k^N,\a^e)$. 
\end{thm}
   Here, we note that  existence of a prime divisor computing the minimal log discrepancy
   is not known in general in positive characteristic, 
   since  existence of resolutions of singularities is not known.
   
   As Musta\cedilla{t}\v{a}-Nakamura's conjecture holds for any pair $(\AA_\CC^2, \a^e)$,
   we obtain the following:
   
\begin{cor} Let $k$ be an algebraically closed field of characteristic $p>0$.
   Musta\cedilla{t}\v{a}-Nakamura's conjecture  holds  for pairs $(\AA_k^2,\a^e)$ in 
   $\mathcal   M_k$.
\end{cor}

By the fact that 
 Musta\cedilla{t}\v{a}-Nakamura's conjecture holds for any pair $(\AA_\CC^N, \a^e)$,
such that $\a_i$'s are all monomial ideals and  by some additional discussions
we obtain the following:

\begin{cor}\label{monomial}  For an algebraically closed field of characteristic $p>0$,
    in the set of pairs $(\AA_k^N, \a^e)$
    such that $\a_i$'s are all monomial ideals, Musta\cedilla{t}\v{a}-Nakamura's conjecture
    and ACC conjecture hold.
\end{cor}
  About log canonical threshold we obtain the following:
\begin{thm}\label{lct}
     Proposition \ref{shibata2} holds in $\mathcal L_k$ for an algebraically closed 
   field of positive characteristic.
   I.e., let $k$ be an algebraically closed field of positive characteristic.
     Let $\m$ be the maximal ideal defining the origin $0\in \AA_k^N$.
    Then, for  a positive integer $\mu$,
    there is a positive integer $L_{N,\mu}$ depending only on $N$ and $\mu$ such that for every 
    ${\m}$-primary ideal $\a$ such that  
       ${\m}^\mu\subset \a$ and $(\AA_k^N,\a)\in \mathcal L_k$, 
    there is some prime divisor $E$ over $\AA_k^N$ that computes $\mathrm{lct}(0;\AA_k^N,\a)$  
    and satisfies  $k_E\le L_{N,\mu}$.
\end{thm}

As a corollary, we obtain the following:

\begin{cor}\label{zhu}
   Let $\m$ be as above.
   Then for a pair $(\AA_k^N, \a)\in \mathcal L_k$ with $\m$-primary ideal $\a$
   there exists a prime divisor over $\AA_k^N$ computing  $\lct(0;\AA_k^N, \a)$.

\end{cor}
  Note that in positive characteristic case, existence of prime divisors computing the log
  canonical threshold is not known in general, as existence of resolutions of singularities is not
  known.

This paper is organized as follows:
In the second section, we establish the basic notions of lifting of polynomials 
to characteristic $0$.
In the third section, we interpret Musta\cedilla{t}\v{a}-Nakamura's conjectures into
a conjecture in terms of jet schemes.
In the forth section, as applications of lifting to characteristic $0$, we prove the  theorems
and also pose a problem whose affirmative answer would give descent of the statements to
the whole set of the pairs of positive characteristic. 
We also show some applications of the affirmative answer of the problem.
In the appendix, the proof of Proposition \ref{shibata2} for characteristic $0$ 
is given  by Kohsuke Shibata.
Our result on log canonical threshold is based on this.

\vskip.5truecm
\noindent
{\bf Acknowledgement.} The author would like to express her hearty thanks to 
Kazuhiko Kurano for simplifying the proof of Proposition \ref{kurano} and Kohsuke Shibata
for helpful comments on the preliminary version of this paper  and 
for providing with the proof of Proposition \ref{shibata2} for characteristic 0
(his proof is included in the appendix of this paper).
She is also grateful to Mircea Musta\cedilla{t}\v{a} for his  comments and
an application of the  descent to whole the set of pairs under the assumption
that Problem \ref{problem} is affirmatively solved
(it is Proposition \ref{mustata}.)
The author also thanks the members of Singularities Seminar at Nihon University for the constant 
encouragements. 

\vskip1truecm
\section{Preliminaries on lifting to characteristic $0$}
\noindent
In this section, we do not assume the algebraic closedness of the base fields.

We introduce a method ``lifting to characteristic $0$" which constructs  objects in 
characteristic $0$ from objects in positive characteristic.

\begin{defn}\label{defofmod0}
  Let $S$ be an integral domain of characteristic 0, i.e., the canonical  homomorphism 
  $\ZZ\to S$ of rings is injective.
  For a prime number $p\in \ZZ$, we denote the canonical projection 
  by $\Phi_p:S\to S\otimes_\ZZ \ZZ/(p)$.
 
\begin{enumerate}
\item (Basic case) For $\tilde f\in S$ and $f\in S\otimes_\ZZ \ZZ/(p)$,
   we say that $\tilde f$ is a {\sl lifting to characteristic $0$ (or just a lifting) } of $f$,
   if $\Phi_p(\tilde f)=f$.
   In this case we also write $\tilde f (\mod p)=f$.
   
\item Let ${\bf f}=\{f_1,f_2,\ldots, f_r\} $ be a set of elements of an integral domain $R$
       of characteristic $p>0$.
       Let  ${\bf \tilde f}=\{\tilde f_1,\tilde f_2,\ldots, \tilde f_r\} $ be a set of elements of an
       integral domain $\widetilde R$ 
       of characteristic 0.
       
       We say that ${\bf \tilde f}$ is a {\sl lifting to characteristic $0$ (or just a lifting)} of 
       ${\bf f}$ and
       write ${\bf \tilde f}(\mod p)={\bf f}$, if the following holds:
 \begin{enumerate}
 \item there exists a subring $S\subset \widetilde R$ and an injective
             homomorphism $\iota:S\otimes_\ZZ \ZZ/(p)\hookrightarrow R$ of rings;
 \item  Identify the ring  $S\otimes_\ZZ \ZZ/(p)$ and its image by the injection 
             $\iota$. 
 The inclusions ${\bf \tilde f}\subset S$ and ${\bf f}\subset S\otimes_\ZZ \ZZ/(p)$ hold 
            with the following relations:
                       $$\tilde f_i (\mod p)=f_i\ \ \mbox{for\ every}\ \ i=1,2,\ldots, r.$$ 
  \end{enumerate}     
  In this case, we call $S$ a {\sl subring supporting } the lifting  ${\bf \tilde f}(\mod p)={\bf f}$.

\item Let $\a$ be an ideal of an integral domain $R$ of characteristic $p>0$.
           Let $\widetilde\a$ be an ideal of an integral domain $\widetilde R$ of characteristic
           0.
           We say that $\widetilde\a$ is a {\sl lifting to characteristic $0$ (or just a lifting)} of $\a$ 
           and write 
           $$\widetilde \a(\mod p)=\a$$ 
           if there exist systems of generators 
           ${\bf f}=\{f_1,\ldots,f_r\} $ of $\a$ and $\tilde{\bf f}=\{\tf_1,\ldots,\tf_r\}$ of 
           $\ta$,
           respectively, such that 
           $${\bf \tilde f}(\mod p)={\bf f}.$$ 
           If $S$ is a subring supporting the lifting ${\bf \tilde f}(\mod p)={\bf f}$, 
           we also call $S$ a subring supporting the lifting $\widetilde \a(\mod p)=\a$.

\end{enumerate}

\end{defn}

\begin{rem}\label{remideal}
\begin{enumerate}
\item
 Note that a lifting of an ideal is not unique and the height of the ideal is not preserved by a lifting. 
We also should note that a lifting of a prime ideal is not necessarily prime.
\item
 When we define $\Phi_p:S\to \sp$
           as the canonical projection,
           it is obvious that the inclusion
$\Phi_p(\widetilde\a\cap S)\supset\a\cap (\sp)$ holds.
In general the equality does not hold.


 \end{enumerate}
\end{rem}

      we obtain the following lifting to characteristic $0$.

\begin{prop}\label{kurano}
       Let $\alpha=\{\alpha_1,\alpha_2,\ldots,\alpha_n\}$ be a set of finite elements of  a field $k$ 
       of characteristic $p>0$.
      Then, there is a subset 
     $\widetilde\alpha=
       \{\widetilde \alpha_1,\widetilde  \alpha_2,\ldots, \widetilde \alpha_n\} \subset \CC$ 
       such that 
       $\widetilde\alpha\ (\mod p)=\alpha$.
\end{prop}

\begin{proof}
    Considering the subring $\zp[\alpha_1,\ldots, \alpha_n]\subset k$,
    we obtain canonical surjections:
    $$R:=\ZZ[Y_1,\ldots,Y_n]\stackrel{\psi}\longrightarrow S:=\zp[Y_1,\ldots,Y_n]
    \stackrel{\varphi}\longrightarrow \zp[\alpha_1,\ldots,\alpha_n]$$
    with $Y_i \mapsto \alpha_i$.
    Let $P:=Ker \varphi \subset S$ and $Q:=Ker \varphi\cdot\psi \subset R$,
    then these are prime ideals in the regular rings.
    Therefore $R_Q$ and $S_P$ are also regular local rings.
    Hence we obtain
    \begin{enumerate}
       \item[] $\tf_1,\ldots, \tf_c, p \  (\in Q)$ which form a regular system of parameters of $R_Q$,
       \item[] $f_1,\ldots,f_c\ (\in P)$ which form a regular system of parameters of $S_P$ and
       \item[] $\psi(\tf_i)=f_i $ for $i=1,\ldots, c$.
    \end{enumerate}
    Then, $R_Q/(\tf_1,\ldots,\tf_c)R_Q$ is also a regular local ring, in particular it is an
    integral domain.
    Consider the homomorphism $\ZZ_{(p)}\to R_Q $ and a regular sequence 
    $f_1,\ldots,f_c \subset R_Q\otimes_\ZZ\zp=
    R_Q\otimes_{\ZZ_{(p)}}\ZZ_{(p)}/(p)\ZZ_{(p)}$,
    we obtain that $R_Q/(\tf_1,\ldots,\tf_c)R_Q$ is flat over $\ZZ_{(p)}$
    by \cite[Corollary, p.177]{mats}.
    In particular, the homomorphism $\ZZ \to R_Q/(\tf_1,\ldots,\tf_c)R_Q$ is injective,
    which implies the ring $R_Q/(\tf_1,\ldots,\tf_c)R_Q$ is of characteristic $0$.
    
    As $PS_P=(f_1,\ldots,f_c)S_P$, there exists $h\in S\setminus P$ such that 
    $PS_h=(f_1,\ldots,f_c)S_h$.
    Take $\th\in R\setminus Q$ such that $\psi(\th)=h$ and let
    $$\Sigma:=R_\th/(\tf_1,\ldots,\tf_c)R_\th,$$
    then it is also an integral domain and of characteristic $0$.
    Now noting that $S/P=\zp[\alpha_1,\ldots, \alpha_c]$, we obtain
    $$\Sigma\otimes_\ZZ\zp=\frac{R_\th/(\tf_1,\ldots,\tf_c)R_\th}
    {p({R_\th/(\tf_1,\ldots,\tf_c)R_\th})}
    =S_h/(f_1,\ldots,f_c)S_h$$
    $$=(S/P)_h\subset 
    \zp(\alpha_1,\ldots,\alpha_c) \subset k.$$
    By the surjection $\Phi_p:\Sigma\to \Sigma\otimes_\ZZ\zp$
    we can take $\widetilde{\alpha}_1, \ldots, \widetilde{\alpha}_n\in \Sigma$
    corresponding to $\alpha_1,\ldots, \alpha_n\in \Sigma\otimes \zp$.
    
     By the definition of $\Sigma$, the field $K_0:=Q(\Sigma)$ of fractions of $\Sigma$ is
     a finitely generated field extension of $\QQ$.
     Then, by Baby Lefschetz Principle (see, for example, \cite{tao}), 
     there is an isomorphism into the subring:
     $$\phi: K_0 \stackrel{\sim}\longrightarrow \phi(K_0)\subset \CC.$$
     Then we obtain 
     $\widetilde \alpha=\{\widetilde \alpha_1,\ldots, \widetilde \alpha_n\}\subset \CC$
      and $\widetilde \alpha\ (\mod p)=\alpha$.
      As is seen, $\Sigma\subset \CC$ is a subring supporting the lifting.
\end{proof}

\begin{prop}\label{polynomial1}
  Let $k$ be a field of characteristic $p>0$.
  Then, the following hold:
\begin{enumerate} 
\item[(i)]
   For  a  subset  ${\bf f}=\{f_1,\ldots, f_n\}\subset k[X_1,X_2,\ldots, X_N]$, 
  there exists a set $\widetilde{\bf f}=
  \{\widetilde f_1,\ldots, \widetilde f_n\}$ in $\CC[X_1,X_2,\ldots, X_N]$
  such that $\widetilde{\bf f} \ (\mod p)=\bf f$.
  
\item[(ii)]  
  For an  ideal $\a\subset k[X_1,X_2,\ldots, X_N]$, there exists an 
   ideal $\ta\subset \CC[X_1,X_2,\ldots, X_N]$ 
  such that $\ta\subset ( X_1,X_2,\ldots, X_N)$ and $$\ta\ (\mod p)=\a.$$ 
  
  In this case we have in general $$\height \ta\geq \height \a.$$
  Here, we define the height of the unit ideal of a ring $R$ as $\dim R+1$.

 \end{enumerate} 
  \end{prop}
  
 \begin{proof} 
   For a proof of (i), let ${\alpha}=\{\alpha_i\}_{i\in I}\subset k$ be a finite set of coefficients 
   of $f_1,\ldots, f_n$
   containing all nonzero coefficients of them.
   By Proposition \ref{kurano}
   we can take a lifting  $\widetilde \alpha=\{\widetilde \alpha_i\}_{i\in I}\subset \CC$ to characteristic $0$ 
   of $\alpha$.
   When $f_j$ ($j=1,\ldots, n$) is represented as 
   $$f_j=\sum_{e\in \ZZ_{\geq 0}^N }\alpha_{j_e}X^e, \ \ \mbox{where}\ \ X^e=X_1^{e_1}\cdots X_N^{e_N},$$
   we define 
   $$\widetilde f_j=\sum_{e\in \ZZ_{\geq 0}^N }\widetilde \alpha_{j_e}X^e.$$
   Then we have 
  $\widetilde{\bf f} \ (\mod p)=\bf f$.
  When a subring supporting $\widetilde{\alpha}\ (\mod p)=\alpha_i$
  is $\Sigma$, then we can take $\Sigma[X_1,\ldots, X_N]$ as a 
  subring supporting the lifting $\widetilde{\bf f} \ (\mod p)=\bf f$.

    For a proof of (ii), let ${\bf f}=\{f_1,\ldots, f_n\}\subset k[X_1,X_2,\ldots, X_N]$
    be a system of generators of $\a$. 
    Then, by (i) we obtain  a subset $\widetilde{\bf f}=
  \{\widetilde f_1,\ldots, \widetilde f_n\}$ in $\CC[X_1,X_2,\ldots, X_N]$
  such that $\widetilde{\bf f} \ (\mod p)=\bf f$.
  Then by definition, we have a lifting $\ta (\mod p)=\a$.

  Under the notation in the proof of (i),
  a subring supporting the lifting $\ta (\mod p)=\a$
  is $S:=\Sigma[X_1,\ldots, X_N]$.
  
  Denote the fields of fractions of $\Sigma\otimes\zp$ and of $\Sigma$ by $Q$ and $\wq$,
  respectively.
   Let $\a_Q\subset Q[X_1,\ldots,X_N]$ be the ideal generated by $f_1,\ldots,f_n$.
    Let $\ta_\Sigma\subset \Sigma[X_1,\ldots,X_N]$ and $\ta_\wq\subset \wq[X_1,\ldots,X_N]$
    be the ideals generated by $\tf_1,\ldots,\tf_n$ in each ring.
    Then $\a$ is the extension of $\a_Q$ and $\ta$ is the extension of $\ta_\wq$.
    As the field extensions $Q\hookrightarrow k$ and $\wq\hookrightarrow \CC$
    are faithfully flat, the ring extensions
    $Q[X_1,\ldots,X_N]\hookrightarrow k[X_1,\ldots,X_N]$ and 
    $\wq[X_1,\ldots,X_N]\hookrightarrow \CC[X_1,\ldots,X_N]$ are also faithfully flat,
    in particular flat.
    Then, by   \cite[Theorem 15.1]{mats},
     we have 
        $$\height \a_Q=\height(\a_Q(k[X_1,\ldots,X_N]))=\height \a,$$
         $$\height \ta_\wq=\height(\ta_\wq(\CC[X_1,\ldots,X_N]))=\height \ta.$$

Therefore, for the proof of the lemma, we have only to show 
   $$\height \a_Q\leq\height\ta_\wq.$$

   Actually let  $\pi: Z(\ta_\Sigma)\to \spec\Sigma$ be the restriction of  the flat morphism\\
   $\spec \ZZ[X_1\ldots,X_N]\to \spec\ZZ$.
   Let $$\Phi'_p:\Sigma[X_1,\ldots,X_N]\to \Sigma\otimes\zp[X_1,\ldots,X_N]\hookrightarrow Q[X_1,\ldots,X_N]$$ 
   be the composite of the canonical projection and the inclusion.
   Then, the ideal $(\Phi'_p(\ta_\Sigma))\subset Q[X_1,\ldots,X_N]$ generated by the image $\Phi'_p(\ta_\Sigma)$
    defines a special fiber of $\pi$ over 
   $\spec Q$.   
   On the other hand, $\ta_\wq$ defines the generic fiber of $\pi$ over
   $\spec \wq$.
   Then, by \cite[Theorem 15.1]{mats}, we have
   $$\dim Z(\ta_\wq)\leq \dim Z(\Phi'_p(\ta_\Sigma)).$$
   Here, as $\spec \Sigma[X_1,\ldots,X_N]\to \spec \Sigma$ is flat, we have the inequality of codimensions, i.e.,
   $$\height \ta_\wq\geq\height (\Phi'_p(\ta_\Sigma)),$$
   where the right hand side is not less than $\height \a_Q$ because 
   of the inclusion \\
   $(\Phi'_p(\ta_\Sigma))\supset \a_Q$.
   (Note that the inclusion is not equal in general.)
   Now we obtain
   $$\height \ta_\wq\geq\height \a_Q,$$ 
   which completes the proof of the equality
   $$\height\ta\geq\height\a.$$

\end{proof}

   
  
%




\vskip1truecm
\section{The conjectures and their relations}

\noindent
   In this section we study a pair $(A, \a^e)$ consisting of a nonsingular 
   affine variety $A$ of dimension $N$  defined
   over an algebraically closed field of arbitrary characteristic
      and a ``multiideal" $\a^e=\a_1^{e_1}\cdots\a_s^{e_s}$ on $A$  with the exponent 
     $e=\{e_1,\ldots,e_s\}\subset \RR_{>0}$.
     First we give a basic definition.

\begin{defn}\label{toric}
    We say that  $E$ is a {\sl prime divisor over} a variety $A$, 
   if there is a birational morphism $A'\to A$ with normal  $A'$ such that
   $E$ is a prime divisor on $A'$.
   
   We say that a prime divisor $E$ over $A=\AA_k^N$ is a {\sl toric prime divisor},
   if there is a vector $w=(w_1,\ldots,w_N)\in \ZZ_{\geq0}^N$ such that
   $$\val_E(f)=\min \{\langle w, u \rangle \mid X^u \in f\},$$ 
   where $X^u \in f$ means that the monomial $X^u $ appears in $f$ 
   with a nonzero coefficient and $\langle w,u \rangle=\sum_i w_iu_i$.
  
   Note that in this case, $E$ is a divisor on a toric variety $A'$ birational to $\AA_k^N$
   such that $E$ corresponds to the one dimensional cone $\RR_{\geq0}w$ in 
   the defining fan of $A'$.

\end{defn}

\begin{defn}\label{defoflogcano}
    The log discrepancy of  a pair $ (A, \a^e)$    
    at a prime divisor $E$ over $A$ is 
    defined as 
$$a(E; A,\a^e):=k_E-\sum_{i=1}^s e_i\val_E\a_i+1,$$
      where $k_E$ is the coefficient of the relative canonical divisor 
      $K_{\overline{A}/A}$ at $E$.\\
Here $\varphi:\overline{A}\to A$ is a 
  birational morphism with normal $\overline{A}$ 
 such that $E$ appears on 
   $\overline{A}$. 

\end{defn}

\begin{defn}\label{defofmld}
    Let $(A,\a^e)$ be a pair  and $x\in A$ a closed point.
    Then the {\sl minimal log discrepancy} is defined as follows:
    
 \begin{enumerate}
\item When $\dim A\geq 2$,
    $$\mld (x; A, \a^e)
    =\inf\{ a(E;A, \a^e) \mid  E :\ {\operatorname { prime\ divisor \ with \ the\ 
     center}}\ \ x \} .$$ 
\item When $\dim A=1$, define $\mld (x; A,\a^e)$ 
     by the same definitions as above if the right hand side
     of the above definition 
     is non-negative and otherwise define $\mld(x; A, \a^e)= -\infty$.
\end{enumerate}
\end{defn}

\begin{defn} 
    Let $A$, $N$,  and $e$ as above and $x\in A$ a closed point.
    We say that a prime divisor $E$ over $A$ with the center $\{x\}$ 
    computes $\mld(x; A, \a^e)$,
    if 
    
    $$\left\{\begin{array}{ll}
       a(E; A, \a^e)=\mld(x; A, \a^e),& \mbox{when}\ \ \mld(x; A, \a^e)\geq 0\\
       a(E; A, \a^e)<0,  &    \mbox{when}\ \ \mld(x; A, \a^e)=-\infty\\ 
    \end{array}\right. $$
\end{defn}

\begin{rem} If there is a log resolution of $(A, \a^e)$ in a neighborhood of $x$,
   or if $e$ is a set of rational numbers, then a prime divisor computing 
   $\mld(x; A, \a^e)$ exists.
   Otherwise, the existence of such a divisor is not known in general.
\end{rem}

     Now we are going to interpret the conjecture $(M_{N,e})$ 
     in terms of jet schemes.
     For that we introduce the notion of jet schemes briefly. 
     For basic properties on jet schemes can be referred in  \cite{EM}, \cite{icr}, \cite{ir2}.
     
\begin{defn}
  
 Let \( X \) be a scheme of finite type over a field $k$ and  $k'\supset k$ a field extension.
For  \( m\in \ZZ_{\geq 0} \) a \( k \)-morphism \( \spec k'[t]/(t^{m+1})\to X \) is called an  {\it{\( m \)-jet}} of \( X \) and 
 \( k \)-morphism \( \spec k'[[t]]\to X \) is called an {\it {arc}} of \( X \).
\end{defn}


Let 
 \( X_{m} \) be the {\it space of \( m \)-jets} or  the $m$-{\it{jet scheme}}
   of  \( X \). 
   It is well known that $X_m$ has a scheme structure of finite type over $k$.
   There exists the projective limit $$X_\infty:=\lim_{\overleftarrow {m}} X_m$$
   and it is called the {\it space of arcs} or the {\it space of $\infty$-jet} of $X$.
   Every point $\spec k'\to X_\infty$ corresponds to an arc
   $\spec k'[[t]]\to X$.

\begin{defn}
    Denote the canonical truncation morphisms induced from $k[[t]]\to k[t]/(t^{m+1})$ 
    and $k[t]/(t^{m+1})\to k$  by
    $\psi_m: X_\infty\to X_m$ and $\pi_m: X_m\to X$, respectively.
    In particular we denote the  morphism  $\psi_0=\pi_\infty : X_\infty \to X$ by $\pi$.
    We also denote the canonical truncation morphism $X_{m'} \to X_m$ $(m'> m)$ 
    induced 
    from $k[t]/(t^{m'+1})\to k[t]/(t^{m+1})$ by
    $\psi_{m', m}$.
 
\end{defn}     
 

\begin{defn} We define the subset {\sl contact locus} in the space of arcs as follows:
       $$\cont^{\geq m}(\a)=\{\gamma \in X_\infty \mid \ord_\gamma(\a):=\ord_t  
       \gamma^*(\a)\geq m\},$$
       where $\gamma^*:\o_X\to k'[[t]]$
      is the  homomorphism of rings corresponding to the arc $\gamma: \spec k'[[t]]\to X$.

     By this definition, we can see that
    $$\cont^{\geq m}(\a)=\psi_{m-1}^{-1}(Z(\a)_{m-1}),$$
     where  $Z(\a)$ is the closed subscheme defined by the ideal $\a$ on $X$.

\end{defn}

\begin{exmp}\label{equation}
     Let $Z\subset \AA^N=:A$ be a closed subscheme of affine $N$-space
    $\AA^N=\spec k[X_1,\ldots,X_N]$ defined over a field $k$ with the defining ideal 
    $\a\subset k[X_1,\ldots,X_N]$.
    Assume $\a$ is generated by $f_1,\ldots, f_r\in k[X_1,\ldots,X_N]$.
    We define polynomials $$F_i^{(j)}\in k[X_\ell^{(q)}\mid 1\leq \ell\leq N, 0\leq q \leq j]$$
    so that 
    $$f_i\left(\sum_{q\geq 0} X_1^{(q)}t^q, \sum_{q\geq 0} X_2^{(q)}t^q, \ldots, \sum_{q\geq 0} X_N^{(q)}t^q\right)=
    F_i^{(0)}+F_i^{(1)}t +\cdots + F_i^{(j)}t^j+ \cdots.$$
    
    Then the $m$-jet scheme $Z_m$ is defined in $(\AA^N)_m\simeq\AA^{N(m+1)}$
    by the ideal of $k[X_\ell^{(q)}\mid 1\leq \ell\leq N, 0\leq q \leq m]$
    generated by
    $$F_i^{(j)}\ \  (i=1,\ldots, r, j=0, 1,\ldots, m).$$
    Here, we note that if all coefficients of $f_i$'s are in a subring $\Sigma\subset k$,
    then all coefficients of   $F_i^{(j)}$'s are also in $\Sigma$.  
    (This fact will be used in the proof of Lemma \ref{localjet}.)
    
    The fiber $\pi_m^{-1}(0)$ of the origin $0\in \AA^N=A$ 
    by the truncation morphism $\pi_m: A_m\to A$ is
    defined in \\
    $k[X_\ell^{(q)}\mid 1\leq \ell\leq N, 0\leq q \leq m]$ by 
    $$X_1^{(0)},\ldots,X_N^{(0)}.$$
  
\end{exmp}
          
\begin{prop}[\cite{EM} for characteristic $0$, \cite{ir2} for arbitrary characteristic] 
\label{formula}
  Let $A$ be a nonsingular variety  defined over an algebraically closed field of
  arbitrary characteristic, $0\in A$ a closed point, $e=\{e_1,\ldots, e_s\}$
  a finite subset of $\RR_{>0}$ and $\a^e$ as  
  in the beginning of this section. 
  Then, we obtain the  following formula:
 
 \begin{equation}\label{formula2}
    \mld(0; A, \a^e) \ \ \ \ \ \ \ \ \ \ \ \ \ \ \ \ \ \ \ \ \ \ \ \ \ \ \ \ \ \ \ \ \ \ \ \ \ \ \ \ \ \ \ \ \ \ \ \ \ \ \ \ \ \ \ \ \ \ \ \ \ \ \ 
\end{equation} 
 $$   =\inf_{m\in \ZZ_{\geq0}^s}\left\{\codim(\cont^{\geq m_1}(\a_1)\cap
 \cdots\cap\cont^{\geq m_s}(\a_s)\cap \pi^{-1}(0), A_\infty)-\sum_{i=1}^s e_i m_i
 \right\}.$$
\end{prop}     
     
\begin{defn} Under the same notation as above. 
     Define the function $s_m(0;A, \a^e)$ on $m=(m_1, m_2,\ldots, m_s)\in \ZZ_{\geq0}^s$
     as follows:
$$s_m(0; A, \a^e)=\codim\left(\cont^{\geq m_1}(\a_1)\cap
 \cdots\cap\cont^{\geq m_s}(\a_s)\cap \pi^{-1}(0), A_\infty\right)-\sum_{i=1}^s e_i m_i
 $$
 
\end{defn}     

\begin{rem}\label{concreteequation}
   For a given $m=(m_1,\ldots,m_s)\in \ZZ_{\geq0}^s$,
   denote $m'_i=m_i-1$ and assume that $m'_1=\max\{m'_i\}$ to simplify
   the notation.
   By the definition of $\codim (-, A_\infty)$ (\cite[Definition 3.2]{DEI}),
   we have the following: 
\begin{equation}\label{sm}   
     s_m(0; A, \a^e)\ \ \ \ \ \ \ \ \ \ \ \ \ \ \ \ \ \ \ \ \ \ \ \ \ \ \ \ \ \ \ \ \ \ \ \ \ \ \ \ \ \ \ \ \ \\  \ \ \ \ \ \ \ \ \ \ \ \ \ \ \ \ \ \ \ \ \ \ \ \ \ \ 
\end{equation}     
 $$    =Nm_1-\dim \left( Z(\a_1)_{m'_1 }\cap
     \psi_{m'_1 ,m'_2 }^{-1}(Z(\a_2)_{ m'_2 })\cap\cdots\cap 
     \psi_{m'_1 ,m'_s }^{-1}(Z(\a_s)_{ m'_s })\cap\pi_{m'_1}^{-1}(0)      \right)-
     \sum_{i=1}^s e_i m_i,
$$  
    where, we define $Z(\a_i)_{-1}=A$ as a convention.
   Later on, we will use this expression of $s_m(0; A, \a^e)$.
   
   If $A=\AA_k^N$, $0\in \AA_k^N$ is the origin and $\a_i$ are generated by $f_{ij}$ $(j=1,\ldots,
   r_i)$ for each $i=1,\ldots, s$, 
   then by Example \ref{equation},
   $$Z(\a_1)_{m'_1 }\cap
     \psi_{m'_1 ,m'_2 }^{-1}(Z(\a_2)_{ m'_2 })\cap\cdots\cap 
     \psi_{m'_1 ,m'_s }^{-1}(Z(\a_s)_{ m'_s })\cap\pi_{m'_1}^{-1}(0)$$
     is defined on $\AA^{Nm_1}=(\AA^N)_{m'_1}$ by the ideal generated by
     $${F}_{1,\ell}^{(q)}\ (\ell=1,\ldots, r_1, q=0,\ldots, m'_1), 
          \ldots ,
         {F}_{s,\ell}^{(q)}\ (\ell=1,\ldots, r_s, q=0,\ldots, m'_s) \ \ \mbox{and}
         \ \ X_1^{(0)},\ldots, X_N^{(0)}.$$
          Note that these are elements of $k[X_i^{(q)} \mid 1\leq i\leq N, 0\leq q\leq m'_1]$.
\end{rem}

\begin{lem}\label{localjet}
    Let $k$ be an algebraically closed field of characteristic $p>0$.
    Let $\a=\{\a_1,\ldots,\a_s\}$ be a set of ideals of $ k[X_1,\ldots, X_N]$
    such that $\a_i\subset (X_1,\ldots,X_N)$ for every $i=1,\ldots,s$
    and $e=\{e_1,\ldots, e_s\}\subset \RR_{>0}$.
    Let $\ta= \{\ta_1,\ldots,\ta_s\}$ be a lifting to characteristic $0$  of $\a$.

    Then for every $m\in \ZZ_{\geq0}^s$ it follows 
    $$s_m(0; \AA_k^N, \a^e)\leq s_m(0; \AA_\CC^N, \ta^e).$$
\end{lem}

\begin{proof} 
    Let $\a_i$ be generated by $f_{ij}$ $(j=1,\ldots, r_i)$
    and  $\ta_i$ be generated by $g_{ij}$ $(j=1,\ldots, r_i)$ 
    such that   
\begin{equation}\label{generator}    
     \{g_{ij} \mid i, j\} \ (\mod p)=
    \{f_{ij} \mid i, j\}.
\end{equation}    
    
    Take $m\in \ZZ_{\geq0}^s$ and let $m_1=\max \{m_i\}$ to simplify the notation.
    We use the notation in Remark \ref{concreteequation}.
    Then, for every $m'_i:=m_i-1$, the $m'_i$-th jet scheme $Z(\a_i)_{m'_i}$
    is defined by 
    $${F}_{i,j}^{(q)}\ (j=1,\ldots, r_1, q=1,\ldots, m'_i), $$
     where   ${F}_{i,j}^{(q)}$ is defined from $f_{i,j}$ as in Example \ref{equation}.
     In the same way, $Z(\ta_i)_{m'_i}$
    is defined by 
    $${G}_{i,j}^{(q)}\ (j=1,\ldots, r_1, q=1,\ldots, m'_i), $$
    where   ${G}_{i,j}^{(q)}$ is defined from $g_{i,j}$ as well.
     Now by the definition of $F_{i,j}^{(q)}$ and $G_{i,j}^{(q)}$ and (\ref{generator}), 
     we have
     $$\{G_{i,j}^{(q)}\mid i,j, q\}\  (\mod p)=\{ F_{i,j}^{(q)}\mid i,j, q \}.$$ 
        In the formula (\ref{sm}),
        we denote the defining ideal of 
        $$Z(\a_1)_{m'_1 }\cap
     \psi_{m'_1 ,m'_2 }^{-1}(Z(\a_2)_{ m'_2 })\cap\cdots\cap 
     \psi_{m'_1 ,m'_s }^{-1}(Z(\a_s)_{ m'_s })\cap\pi_{m'_1}^{-1}(0) $$
     in $(\AA_k^N)_{m'_1}$
     by $I_{m}(\a)$ and the defining ideal of
           $$Z(\ta_1)_{m'_1 }\cap
     \psi_{m'_1 ,m'_2 }^{-1}(Z(\ta_2)_{ m'_2 })\cap\cdots\cap 
     \psi_{m'_1 ,m'_s }^{-1}(Z(\ta_s)_{ m'_s })\cap\pi_{m'_1}^{-1}(0) $$
     in $(\AA_\CC^N)_{m'_1}$
     by $I_{m}(\ta)$.
        Since, $I_m(\a)$ is generated by 
        $$ {F}_{1,j}^{(q)}\ (j=1,\ldots, r_1, q=1,\ldots, m'_1), 
          \ldots ,
          {F}_{s,j}^{(q)}\ (j=1,\ldots, r_s, q=1,\ldots, m'_s), X_1^{(0)},\ldots, X_N^{(0)}$$
          and $I_m(\a)$ is generated by 
        $$ {G}_{1,j}^{(q)}\ (j=1,\ldots, r_1, q=1,\ldots, m'_1), 
          \ldots ,
          {G}_{s,j}^{(q)}\ (j=1,\ldots, r_s, q=1,\ldots, m'_s), X_1^{(0)},\ldots, X_N^{(0)}.$$
          We obtain that 
          $$I_m(\ta)\ (\mod p)=I_m(\a).$$
       Then by Proposition  \ref{polynomial1}, (ii), we have
 \begin{equation}\label{heightIm}
       \height I_m(\a)\leq \height I_m(\ta).
 \end{equation}      
       On the other hand, (\ref{sm}) gives
       
      $$      s_m(0; \AA_k^N, \a^e)\  =Nm_1-\dim Z(I_m(\a))-
        \sum_{i=1}^s e_i m_i,
    $$ 
        $$      s_m(0; \AA_\CC ^N, \a^e) =Nm_1-\dim Z(I_m(\ta))-
        \sum_{i=1}^s e_i m_i,
    $$      
       Therefore, by (\ref{heightIm})  we obtain 
       $$s_m(0;\AA_k^N,\a^e)\leq s_m(0;\AA_\CC^N,\ta^e).$$       
        \end{proof}

\begin{defn}\label{smcompute}
     Let $A$, $x$, $\a$ and $e$ be as in the beginning of this section.
    We say that $m\in \ZZ_{\geq0}^s$ (or $s_m(x; A, \a^e)$)
     computes $\mld(x; A,\a^e)$ if 
    $$\left\{\begin{array}{ll}
       s_m(x; A, \a^e)=\mld(x; A, \a^e),& \mbox{when}\ \ \mld(x; A, \a^e)\geq 0\\
       s_m(x; A, \a^e)<0,  &    \mbox{when}\ \ \mld(x; A, \a^e)=-\infty\\ 
    \end{array}\right. $$ 
\end{defn}

Now, we state  the conjectures about the minimal log discrepancies.

\begin{conj}[$D_{N,e}$]
     Let $A=\AA_k^N$ be  defined over 
     an algebraically closed field $k$ of arbitrary characteristic, $0$ the origin
      and $e$ a finite subset of $\RR_{>0}$.
    Then, there is a positive integer $\ell'_{N,e}$(depending only on $N$ and $e$) such that
     for every multiideal $\a^e$ with the exponent $e$ on $A$,
     there is $m\in \ZZ^s_{\geq 0}$ $(s=\# e)$ computing $\mld(0; A,\a^e)$
     satisfying  $|m|:=\sum m_i \leq \ell'_{N,e}$.

\end{conj}

  We obtain the equivalence of the conjectures.
   The equivalence 
   on  $\a^e$ of  a special type is proved in \cite{findet} and
  its proof works essentially for the general case, 
  but for the reader's convenience we exhibit the proof here. 

\begin{prop}\label{equiv}
    Let $A$ be an $N$-dimensional nonsingular affine variety defined over 
     an algebraically closed field $k$ of arbitrary characteristic, $0\in A$ a closed point
      and $e$ a finite subset of $\RR_{>0}$.
      Then, 
      we obtain the following:
    \begin{enumerate}
         \item[(i)]  For any multiideal $\a^e$ with the exponent $e$ on $A$,
         there exists a prime divisor $E$ computing $\mld(0;A, \a^e)$
         if and only if there exists $m\in \ZZ_{\geq0}^s$  
        computing $\mld(0;A, \a^e)$;
        \item[(ii)] In the set of pairs $(A, \a^e)$ consisting of 
        nonsingular affine varietie $A$ and  multiideals ideal $\a^e$,
        $M_{N,e}$ holds if and only if $D_{N,e}$ holds.         
   \end{enumerate}
\end{prop}

\begin{proof} 
     We will prove (i) and (ii) together in the following way:
     
 \noindent
 {\bf Step 1.} 
      Assume that a prime divisor $E$ over $A$ computes $\mld(0;A,\a^e)$, 
      then there exists $m\in \ZZ_{\geq0}^s$  computing  $\mld(0;A,\a^e)$.
      
      If moreover $k_E\leq \ell_{N,e}$ holds, then we can take $m$ computing
      $\mld$ such that 
        $$|m|\leq (\ell_{N,e}+1+\max\{e_i\})/\min\{e_i\}.$$
      
\noindent
{\bf Step 2.}
    Assume that $m$ computes   $\mld(0;A,\a^e)$,
    then there exists a prime divisor $E$ over $A$ computing $\mld(0;A,\a^e)$.
    
    If moreover $m$ satisfies $|m|\leq \ell'_{N,e}$,
    then the $E$ obtained above satisfies 
    $$k_E\leq  N-1+\ell'_{N,e}\cdot\max\{e_i\}.$$
    
 \noindent
 [Proof of Step 1]     
      For the divisor $E$ in the statement of Step 1,
        let $m_i:=\val_E \a_i$.
    Then, $$k_E+1-\sum_{i=1}^se_im_i= \mld(0; A,\a^e)\ \ \mbox{or}\ \ <0.$$
    Consider the maximal divisorial set  $C_A(\val _E)\subset A_\infty$ that is the closure of the irreducible subset
    $$\{\gamma\in A_\infty \mid \ord_\gamma=\val_E\}$$
    (see  \cite{DEI} or \cite{isv} for details) .
    Then, as  $m_i=\val_E \a_i$ for every $i$,  $C_A(\val _E)$ 
    is contained in the closed subset 
    $$C(m):=\cont^{\geq m_1}(\a_1)\cap \cont^{\geq m_2}(\a_2)\cap\cdots\cap
    \cont^{\geq m_s}
    (\a_s)\cap\pi^{-1}(0)$$
    of $A_\infty$.
    Then, $$\codim (C_A(\val _E), A_\infty)\geq \codim (C(m), A_\infty).$$
    As $s_m(0;A,\a^e)=\codim (C(m), A_\infty)-\sum e_im_i$ and 
    $\codim (C_A(\val _E), A_\infty)=k_E+1$  
     (it is proved in \cite{DEI} for characteristic $0$ and in \cite{ir2} for positive characteristic)
       we obtain
 \begin{equation}\label{smke}
    s_m(0;A,\a^e)\leq k_E+1-\sum e_im_i,
    \end{equation}
    As $E$ computes $\mld(0;A,\a^e)$, the right hand side of the above
    inequality is either $\mld(0;A,\a^e)$ or negative.
    This shows that $m$ computes $\mld(0;A,\a^e)$.
    
   Now we assume $k_E\leq \ell_{N,e}$.
   First we consider the case $\mld(0; A,\a^e)\geq0$.
   Then by (\ref{smke}), we obtain
     $$\ell_{N,e}+1-\mld(0;A,\a^e)\geq k_E+1-\mld(0;A,\a^e)=\sum e_im_i\geq \min_i\{e_i\}|m|,$$
   and it shows $$|m|\leq \frac{\ell_{N,e}+1}{\min_i\{e_i\}}.
   $$

   Next we consider the case $\mld(0; A,\a^e)=-\infty$.
   For a prime divisor $E$ computing 
     $\mld(0;A,\a^e)$ 
     we have 
     $$k_E+1-\sum_ie_i\val_E(\a_i)<0.$$
     But  even for some choices of $n=(n_1,\ldots,n_s)\in \ZZ_{\geq0}$ with $n_i\leq \val_E(\a_i)$ $(\forall i)$ the following inequality may hold:
     $$k_E+1-\sum_ie_in_i<0.$$
     Considering this fact, we define
     $${\mathcal D}:=\left\{ E  \left|  \begin{array}{l}\mbox{prime\ divisor\ computing }\mld(0;A,\a^e)\\
                                                                    \mbox{and\ satisfying\ }k_E\leq \ell_{N,e}\\
                                          \end{array}\right.\right\}    $$      
      and let $n\in \ZZ_{\geq0}^s$ attain the minimal value:
      $$\min\left\{ |n|  \ \left| \ \begin{array}{l} E\in {\mathcal D},\  n_i\leq \val_E(\a_i) \  (\forall i)\\
                                                                                  k_E+1-\sum_ie_in_i<0\\
                                                      \end{array}\right.\right\}.$$                               
      Let $E\in \mathcal D$ give the minimal value $|n|$.                                    
     Here, we may assume that $n_1\geq 1$.
     Now, define $n^*=(n^*_1,\ldots, n^*_s)$, so that  $n^*_1=n_1-1$ and $n^*_i=n_i$ $(i\neq 1)$.
     
\noindent
{\bf Claim.}  $s_{n^*}(0;A,\a^e)\geq 0.$

       Considering the definition of $s_{n^*}(0;A,\a^e)$, 
       we have an irreducible component 
       $C\subset C(n^*)$
       such that $$\codim(C, A_\infty)=\codim( C(n^*), A_\infty).$$
       It is known that a finite codimensional irreducible component of the intersection of finite 
       number of contact loci 
       is a maximal divisorial set (\cite[Corollary 3.16]{ir2}), 
       therefore  $C=C_A(q\cdot\val_{E'})$ for some $q\in \NN$ and a prime
       divisor $E'$ over $A$ with the center $\{0\}$.
       Therefore,  as $\codim(C_A(q\cdot\val_{E'}), A_\infty)=q(k_{E'}+1)$
       (it is proved in \cite{DEI} for characteristic $0$ and in \cite{ir2} for positive characteristic)
        we obtain
   \begin{equation}\label{minimal}    
       q(k_{E'}+1)-\sum e_in^*_i=s_{n^*}(0;A,\a^e).
   \end{equation}    
       Assume that this value is negative.
       Then, by (\ref{minimal}) and $q\cdot\val_{E'}(\a_i)\geq n^*_i$, we have
       $$q\left(k_{E'}+1-\sum e_i\val_{E'}(\a_i)\right)\leq q(k_{E'}+1)-\sum e_in^*_i=s_{n^*}(0;A,\a^e)<0,$$
       which implies that $E'$ computes $\mld(0;A,\a^e)$.
       On the other hand, by $n^*_i\leq n_i$ for every $i$,
       we have $C_A(\val_E)\subset C(n)\subset C(n^*)$,
       which yields 
       $$\ell_{N,e}+1\geq k_E+1=\codim (C_A(\val_E),A_\infty)\geq \codim (C(n^*), A_\infty)$$
       $$=\codim (C_A(q\cdot\val_{E'}), A_\infty)=q(k_{E'}+1)\geq k_{E'}+1.$$
       Thus $E'$ computes $\mld(0;A,\a^e)=-\infty$,  
       $k_{E'}\leq \ell_{N,e}$, $ k_{E'}+1-\sum_ie_in^*_i<0$ and
       $|n^*|<|n|$, which is a contradiction to the minimality of $|n|$.
       
       Now, as we showed the claim $s_ {n^*}(0;A,\a^e)\geq 0$ in the above discussion,
       we obtain $q(k_{E'}+1)-\sum e_in^*_i\geq 0$.
       Here, as we see above $k_E+1\geq q(k_{E'}+1)$, we obtain
       $$k_E+1-\sum e_in_i+e_1=k_E+1-\sum e_in^*_i\geq 0,$$
       which yields
       $$l_{N,e}+1+e_1\geq \sum e_in_i\geq\min\{e_i\}|n|.$$
       Now we obtain that $n\in \ZZ_{\geq0}^s$ computes $\mld(0;A,\a^e)$ 
        and satisfies \\
         $|n|\leq (\ell_{N,e}+1+e_1)/\min\{e_i\}$.
        Hence, to conclude the proof of Step 1,
        we can take $$\ell'_{N,e}:=(\ell_{N,e}+1+\max\{e_i\})/\min\{e_i\}.$$ 
        
\vskip.5truecm
\noindent
[Proof of Step 2]
      Assume that $m$ computes the minimal log discrepancy.
   Take a maximal divisorial set $C_A(q\cdot \val_E)\subset C(m)$ which gives  
   $\codim(C(m), A_\infty)$.
   Then 
   $$\codim ( C_A(q\cdot \val_E), A_\infty)=s_m(0;A,\a^e)+\sum e_im_i.$$ 
   Here,  
   we have $\codim (C_A(q\cdot \val_E),A_\infty)=q(k_E+1)$.
   On the other hand, since $C_A(q\cdot \val_E)\subset C(m)$, it follows
   $q\cdot \val_E(\a_i)\geq m_i$.
     Then, we obtain
\begin{equation}\label{kesm}     
     q(k_E+1)-\sum e_i(q\cdot\val_E(\a_i))\leq q(k_E+1)  -\sum e_im_i=s_m(0;A,\a^e),
\end{equation}     
    which implies that $E$ computes $\mld(0; A,\a^e)$ in both cases $\mld \geq 0$ and
    $\mld=-\infty$.
\vskip.5truecm
   Now, assume $m\in \ZZ_{\geq0}^s$ computes $\mld(0;A,\a^e)$ and satisfies 
     $|m|\leq \ell'_{N,e}$.
     Take a prime divisor $E$ over $A$ as above.
     
     First we consider the case $\mld(0; A,\a^e)\geq0$.
     Then (\ref{kesm}) implies $q=1$ and the following:
\begin{equation}\label{boundbyN}
     k_E+1-\sum e_i\val_E(\a_i)\leq k_E+1-\sum e_im_i=s_m(0;A,\a^e)=\mld(0;A,\a^e)\leq N,
\end{equation}     
     which yields that $E$ computes $\mld(0;A,\a^e)$ and 
     $$k_E\leq N-1+\ell'_{N,e}\cdot\max\{e_i\}.$$ 
      
      Next we consider the case $\mld(0;A,\a^e)=-\infty$,
     then by (\ref{kesm}), we have
     $$q\cdot(k_E+1-\sum e_i \val_E(\a_i)) \leq q(k_E+1)-\sum e_im_i=s_m(0;A,\a^e)<0,$$
     which also yields that $E$ computes $\mld(0;A,\a^e)$
     and $$k_E\leq \ell'_{N,e}\cdot\max\{e_i\}-1     .$$ 
     In the both cases we obtain a bound
     $$k_E\leq N-1+\ell'_{N,e}\cdot\max\{e_i\}.$$ 
 \end{proof}
%

   Up to now we discussed about minimal log discrepancies of a pair,
   In the rest of this section we  discuss about log canonical threshold
   which also evaluates ``goodness" of the singularities of a pair.
   
  To show the statements on log canonical threshold, we remind us the definition of
   log canonical threshold $\lct(A, \a)$ for a nonsingular variety $A$ and a coherent ideal
   sheaf $\a$ on $A$ and some basic properties.

\begin{defn} Let $A$ be a nonsingular variety defined over an algebraically closed field
      and $\a$ a nonzero coherent ideal sheaf on $A$.
       {\sl Log canonical threshold} $\lct(x; A, \a)$ at a closed point $x\in A$ is defined as follows:
       $$\lct(x;A,\a):=\sup\{r\in \RR_{\geq0} \mid (A, \a^r)\ \ \mbox{is\ log\ canonical
       \ at\ } x\}.$$
       Note that 
       $$\lct(x;A,\a)=\inf \left\{ \frac{k_E+1}{\val_E(\a)} \mid E\ \mbox{prime\ divisor\ over\ }
       A \ \mbox{with\ } c_A(E)\ni x \right\}.$$
       Let $E$ be a prime divisor  over $A$ with the center containing $x$.
       We say that  $E$
       {\sl computes}  $\lct(x;A,\a)$ if 
       $$\frac{k_E+1}{\val_E(\a)}=\lct(x;A,\a)$$
       holds.
\end{defn} 

\begin{rem}
\begin{enumerate}
\item
      Log canonical threshold is defined for any klt variety, but in this paper we only 
     work on the origin $0$ on the affine space $A=\AA_k^N$.
\item 
     If $(A, \a)$ has a log resolution $\varphi:\widetilde A\to A$, then 
     $$\lct(0;A,\a)=\min \left\{ \frac{k_E+1}{\val_E(\a)} \mid E\ \mbox{ prime\ divisor\ on\ }
      \widetilde{A}\ \mbox{with}\  \varphi(E)\ni 0\ \right\}.$$
      Therefore, in this case there is a prime divisor over $A$ computing $\lct(0;A,\a)$.
      But, up to this moment, for the base field of characteristic $p>0$
      the existence of a prime divisor computing log canonical threshold is not known in general.

\end{enumerate}

\end{rem}

\begin{prop}[{\cite{must}} for characteristic $0$ and {\cite{zhu}} for positive characterisitic]
\label{lct-by-jet}
     Let $A=\AA_k^N$, $0$ and $\a$ as above, then there is an interpretation of log canonical threshold by
     jet schemes as follows:
     $$\lct(0;A,\a)=\inf_{m\in \NN} \frac{\codim_0(Z(\a)_m, A_m)}{m+1},$$
      where $\codim_0(Z(\a)_m, A_m)=\codim((Z(\a)\cap U)_m, A_m)$ 
      for sufficiently small neighborhood $U$ of $0$.
      Note that the right hand side is constant for sufficiently small $U$ and 
      in particular if $\surd\overline\a=\m$ for the maximal ideal $\m$ defining $0$ in $A$,
      we have  
      $$\codim_0(Z(\a)_m, A_m)=\codim(Z(\a)_m, A_m)=\codim (\cont^{\geq m+1}(\a),
      A_\infty).$$
 \end{prop}

\begin{defn}  We define a function $z_m(0;A,\a)$ on $m\in \NN$ as follows:
 $$z_m(0;A,\a):= \frac{\codim_0(Z(\a)_m, A_m)}{m+1}.$$
    Then, as in Proposition \ref{lct-by-jet}, 
     $$\lct(0;A,\a)=\inf_{m\in \NN}z_m(0;A,\a).$$
\end{defn}

   
\begin{lem}\label{zm-ke}
   Let $k$ be an algebraically closed field of arbitrary characteristic
   and $A=\AA_k^N$ the affine space of dimension $N\geq1$ defined over $k$.
   Let $\m$ be the defining ideal of the origin $0\in A$
   and $\mu$  a positive integer.
   Define the set $\mathcal{I}_\mu$ of ideals as follows:
     $$\mathcal{I}_\mu:=\{\a\ \mbox{ideal\ on}\ A \mid \m^\mu\subset \a\subset \m\}.$$
    Then the following are equivalent:
\begin{enumerate}
   \item[(i)] There
     is a positive integer $L_{N,\mu}$ (depending only on $N$ and $\mu$) such that for every 
          ideal $\a\in \mathcal{I}_\mu$,
    $\lct(0; A, \a)$ is computed by a prime divisor $E$ over $A$  with $k_E\leq L_{N,\mu}$.    
    \item[(ii)] There
     is a positive integer $L'_{N,\mu}$ (depending only on $N$ and $\mu$) such that for every 
          ideal $\a\in \mathcal{I}_\mu$, $z_m(0; A, \a)$ computes $\lct(0; A, \a)$ 
          (i.e.,
    $\lct(0; A, \a)=z_m(0; A, \a)$ holds) for some $m\leq L'_{N,\mu}$.
\end{enumerate}    
\end{lem}   

\begin{proof} First we prove (i) $\Rightarrow$ (ii).
      By the assumption (i), 
      for an ideal $\a\in \mathcal{I}_\mu$
      there is a prime divisor $E$ over $A$ with  $k_E\leq L_{N,\mu}$
       such that $E$ computes $\lct(0;A,\a)$.
      Let $\val_E\a=m+1$, then we have 
\begin{equation}\label{minimalE}      
      \lct(0;A,\a)=\frac{k_E+1}{m+1}\leq \frac{L_{N,\mu}+1}{m+1}.
 \end{equation}     
   Here, as $\val_E\a=m+1$, we have 
    $$C_A(\val_E)\subset \cont^{\geq m+1}(\a).$$
    Then considering the codimensions in $A_\infty$ of the both hand sides,
    we obtain
    $$k_E+1\geq \codim(\cont^{\geq m+1}(\a), A_\infty).$$
    Therefore, 
    $$\frac{k_E+1}{m+1}\geq z_m(0;A,\a).$$
    Here, as the left hand side is $\lct(0;A,\a)$ and the right hand side is not less than $\lct(0;A,\a)$,
     the equality holds:
    $$\lct(0;A,\a)=z_m(0;A,\a)$$ 
     and by (\ref{minimalE}) the value $m$ satisfies
    $$m=\frac{k_E+1}{\lct(0;A,\a)}-1\leq \frac{L_{N, \mu}+1}{\lct(0;A,\a)}-1\leq \frac{L_{N, \mu}+1}{\lct(0;A,\m^\mu)}-1,$$
    where note that the last term does not depend on the choice of $\a$.
  
    Next we prove (ii) $\Rightarrow$ (i).
    Take an ideal $\a\in \mathcal{I}_\mu$.
    As we assume (ii), there exists $m\leq L'_{N,\mu}$ such that 
\begin{equation}\label{computezm}      
      z_m(0;A,\a)=\lct(0;A,\a)
 \end{equation} 
    Take an irreducible component 
\begin{equation}\label{inclusion}    
    C_A(q\cdot\val_E)\subset \cont^{\geq m+1}(\a)
\end{equation}
    such that $\codim  (C_A(q\cdot\val_E),A_\infty)=\codim ( \cont^{\geq m+1}(\a), A_\infty)$.
    Then we have $$\frac{q(k_E+1)}{m+1}=z_m(0;A,\a).$$
    On the other hand, by the inclusion (\ref{inclusion}), it follows
    $q\cdot\val_E\a\geq m+1$.
    Then, $$\lct(0;A,\a)\leq \frac{k_E+1}{\val_E\a}=\frac{q(k_E+1)}{q\cdot\val_E\a}
    \leq \frac{q(k_E+1)}{m+1}=\lct(0;A,\a).$$
    Hence, $E$ computes $\lct(0;A,\a)$ and $k_E+1\leq (m+1)\cdot\lct(0;A,\a)\leq (L'_{N,\mu}+1)N$.
    Thus we obtain a uniform bound
    $$k_E\leq  (L'_{N,\mu}+1)N-1.$$

\end{proof}
   

\vskip1truecm
\section{Proofs of the main results and a problem}

%
%
%
\noindent     
In this section, we will prove the main theorems stated in the section one
and give a problem whose affirmative answer would give descent of the statements
to the whole set of pairs $(\AA_k^N,\a^e)$ of positive
characteristic.

\begin{lem}\label{keylemma}
   Let $k$ be an algebraically closed field of characteristic $p>0$,
   $(\AA_k^N, \a^e)$ a pair and $E$ a toric prime divisor over $\AA_k^N$ with the center
   at $0$.
   Then, there are a prime divisor $\we$ over $\AA_\CC^N$ with the center at $0$
    and multiideal $\ta^e$
    such that 
     $$k_E=k_\we, \ \ \val_E\a_i=\val_\we\ta_i,\ \ \  \mbox{and} \ \ 
   \ta_i (\mod p)=\a_i\  \mbox{ for\ all\ } i,$$
  where $\ta^e=\ta_1^{e_1}\cdots\ta_s^{e_s}$
   and $\a^e=\a_1^{e_1}\cdots\a_s^{e_s}$.

\end{lem}

\begin{proof} Let the toric prime divisor $E$ over $\AA_k^N$ correspond to the vector $w=(w_1,\ldots,w_N)\in 
   \ZZ_{>0}$.
   Let $\we$ be the toric  prime divisor over $\AA_\CC^N$ corresponding to $w$.
   Then, $$k_E+1=\langle w, {\bf1}\rangle =k_\we+1,$$
   where ${\bf1}=(1,1,\ldots,1)$.

   For our lemma, it is sufficient to construct a lifting $\ta_i$  of $\a_i$ $(i=1,\ldots, s)$ such that 
   $$\val_E\a_i=\val_\we\ta_i.$$
   First we will show that for a polynomial $f\in k[X_1,\ldots, X_N]$ there exists 
   $\tf\in \CC[X_1,\ldots,X_N] $ such that
   $$\tf (\mod p)=f \ \ \mbox{and}\ \ \val_Ef=\val_\we\tf.$$
   Express $$f=\sum_{u\in \ZZ_{\geq0}^N}\alpha_uX^u.$$
   Let $d:=\val_E f=\min_{\alpha_u\neq 0}\langle w,u\rangle$.
   By Proposition \ref{polynomial1}, there exists a lifting $\of=\sum_{u\in \ZZ_{\geq0}}
   \widetilde{\alpha}_uX^u\in \CC[X_1,\ldots, X_N] $ of $f$.
   Divide $\of$ into the sum of two polynomials
   $$\of=\tf_{<d}+\tf_{\geq d},$$
    where $\tf_{<d}=\sum_{\langle w,u\rangle<d}
   \widetilde{\alpha}_uX^u$ and $\tf_{\geq d}=\sum_{\langle w,u\rangle\geq d}
   \widetilde{\alpha}_uX^u$.
    Note that $\tf_{<d} (\mod p)=0$, as $\of (\mod p)=f$ and $\val_E f=d$.
    Define $\tf:=\tf_{\geq d}$.
    Then, we obtain that 
    $$\tf (\mod p)=f\ \mbox{and}$$
     $$\val_\we\tf=\min_{X^u\in \tf}\langle w,u\rangle = d=\val_E f.$$
Next we apply the discussion above to the ideals.
  For each generator $f_{ij}$ of $\a_i$ we
   obtain a lifting $\tf_{ij}$ to characteristic $0$  by the above procedure.
   Here, note that we can take them in a common subring supporting the liftings.
   Let them generate an ideal $\ta_i\subset \wr_0$.
   Then we obtain that
   $$\ta_i (\mod p)=\a_i.$$ 
   If $f_{i1}$ computes $\val_E\a_i$ then $\tf_{i1}$ computes $\val_\we\ta_i$
   by the definition of $\ta_i$ and we obtain
   $$\val_E \a_i=\val_E f_{i1}=\val_\we\tf_{i1}=\val_\we\ta_i,$$
  which completes the proof.
\end{proof}

\vskip.5truecm
\noindent    
{\bf [Proof of Theorem \ref{mainthm}]}
      Let $k$ be an algebraically closed field of characteristic $p>0$.
    We divide our discussion into two cases  (a) $\mld(0; \AA_k^N, \a^e)\geq 0$ and \\
    (b) $\mld(0; \AA_k^N, \a^e)=-\infty$.

    In case (a),  let $\delta:=\mld(0; \AA_k^N, \a^e)\geq 0$.
    Since $( \AA_k^N,\a^e)\in \mathcal M_k$,
    for every $\epsilon>0$ there exists a toric prime divisor $E_\epsilon$ over $\AA_k^N$
    with the center at $\{0\}$ such that
    $$ a(E_\epsilon ;\AA_k^N,\a^e) -\delta   <\epsilon.$$
    Then, by Lemma \ref{keylemma}, there exist  a multiideal $\ta_\epsilon$ in $\CC[X_1,
    \ldots, X_N]$ and a prime divisor $\we_\epsilon$ over $\AA_{\CC}^N$ 
    such that 
\begin{equation}\label{epsilon1}    
    a(E_\epsilon;\AA_k^N,\a^e)= a(\we_\epsilon;\AA_{\CC}^N,\ta_\epsilon^e).
\end{equation}      
   Here, as we assume that $(M_{N,e})$ holds for the base field $\CC$,
   $(D_{N,e})$ also holds for $\CC$ by Proposition \ref{equiv}.
   Therefore, there exists $m_\epsilon\in \ZZ_{\geq0}^s$ such that 
   $|m_\epsilon|\leq \ell'_{N,e}$
   and 
\begin{equation}\label{epsilon2}
    s_{m_\epsilon}(0;\AA_{\CC}^N,\ta_\epsilon^e)= \mld(0;\AA_{\CC}^N,\ta_   
    \epsilon^e).
\end{equation}  

     On the other hand, by Lemma \ref{localjet}, we have: 
\begin{equation}\label{epsilon3}     
      s_{m_\epsilon}(0;\AA_{\CC}^N,\ta_\epsilon^e)
     \geq  s_{m_\epsilon}(0;\AA_{k}^N,\a^e).
\end{equation}    
 
   By composing (\ref{epsilon1}), (\ref{epsilon2}) and (\ref{epsilon3}), we have
   $$\delta+\epsilon> a(E_\epsilon;\AA_k^N,\a^e)
         = a(\we_\epsilon;\AA_{\CC}^N,\ta_\epsilon^e)
     \geq\mld(0;\AA_{\CC}^N,\ta_\epsilon^e)$$
         $$=s_{m_\epsilon}(0;\AA_{\CC}^N,\ta_\epsilon^e)
         \geq   s_{m_\epsilon}(0;\AA_{k}^N,\a^e)\geq \delta .$$   
     Therefore we have  a sequence $\{m_\epsilon\}_\epsilon $ 
     \begin{center}
     such that
      $|m_\epsilon|<\ell'_{N,e}$ and $| s_{m_\epsilon}(0;\AA_{k}^N,\a^e)-\delta|<\epsilon$.        
   \end{center}
    Since the set $\{m\in \ZZ_{\geq0}^s\mid |m|\leq \ell'_{N,e}\}$ is finite,
     we have that 
    $m_\epsilon=m$ (constant) for sufficiently small $\epsilon$, 
    and $$s_{m}(0;\AA_{k}^N,\a^e)=\delta.$$
    Therefore, we obtain that this $m$ satisfies $|m|\leq \ell'_{N,e}$ and  computes $\mld(0;\AA_{k}^N,\a^e)$.
    By Proposition \ref{equiv}, this proves $(M_{N,e})$ holds in $\mathcal M_k$
     and also there exists a prime divisor over $\AA_k^N$
    that computes $\mld(0;\AA_k^N,\a^e)$.

   \vskip.5truecm
    In case (b) $\mld(0; \AA_k^N, \a^e)=-\infty$, by the definition of $\mathcal M_k$,
    there exists a toric prime divisor $E$ over $\AA_k^N$
   with the center at $\{0\}$ such that $a(E;\AA_k^N,\a^e)<0$.
    
     Then, by Lemma \ref{keylemma}, there exist a multiideal  
    $\ta$ of $ \CC[X_1,\ldots,X_N]$ and  a prime divisor $\we$ over 
    $\AA_\CC^N$  such that
    $$a(E; \AA_k^N,\a^e)=a(\we; \AA_\CC^N, \ta^e)<0,$$
    which implies $\mld(0;\AA_\CC^N)=-\infty$.
    Since we assume that $(M_{N,e})$ holds over the base field $\CC$,
     $(D_{N,e})$ also holds by Proposition \ref{equiv}
     and therefore 
    there exists $m$ such that $|m|< \ell'_{N,e}$ and $s_m(0; \AA_\CC^N, \ta^e)<0$.
    Here, as is seen in Lemma \ref{localjet}, we obtain
      $$s_m(0; \AA_k^N, \a^e)\leq s_m(0; \AA_\CC^N, \ta^e)<0.$$
    Thus we obtain $m$ such that   $|m|< \ell'_{N,e}$ and computes $\mld(0; \AA_k^N, \a^e)=-\infty$.
    So, in this case $(M_{N,e})$ holds for $k$. $\Box$
 
 \vskip.5truecm
 \noindent
{\bf [Proof of Corollary \ref{monomial}]}
  Let $\a$ be a set of monomial ideals $\{\a_1,\ldots, \a_s\}$ of $k[X_1,\ldots,X_N]$.
  Define a set of monomial ideals $\ta=\{\ta_1,\ldots, \ta_s\}$ of $\CC[X_1,\ldots,X_N]$ 
  such that $\ta_i$ is generated by the same monomials as $\a_i$ for every $i$.
  Then, $$\ta_i (\mod p)=\a_i.$$
  It is known that both $(\AA_k^N, \a_1\cdot\a_2\cdots\a_s)$ and 
  $(\AA_\CC^N, \ta_1\cdot\ta_2\cdots\ta_s)$ have toric log resolutions (\cite{H}). 
  As the Newton polygon of the ideal $\a_1\cdot\a_2\cdots\a_s$ is the same as that of 
  $\ta_1\cdot\ta_2\cdots\ta_s$, we can take a common fan corresponding to
  log resolutions.
  In general, on a log resolution there is a prime divisor computing the minimal log discrepancy.
  In our case all divisors on the log resolutions are toric divisors and they satisfy
  $$\val_E(\a_i)=\val_\we(\ta_i) \ \ \mbox{and}\ \ \ k_E=k_\we,$$
  where $E$ and $\we$ the toric divisors over $\AA_k^N$ and $\AA_\CC^N$, 
  respectively,  corresponding to a common vector $w\in \ZZ_{\geq0}^N$.
  Then, we can see that there is  a toric prime divisor computing $\mld(0; \AA_k^N,\a^e)$ and
  \begin{equation}\label{equal}
  \mld(0; \AA_k^N,\a^e)=\mld(0;\AA_\CC^N,\ta^e).
  \end{equation}
Hence, by Theorem \ref{mainthm}, Musta\cedilla{t}\v{a}-Nakamura's conjecture descends
to the set of the pairs $(\AA_k^N,\a^e)$ with monomial ideals $\a_i$.
On the other hand, the conjecture is known to hold for the pairs $(\AA_\CC^N,\ta^e)$
with monomial ideals $\ta_i$ in characteristic $0$.
Therefore, for the whole set of pairs $(\AA_k^N,\a^e)$ with monomial ideals $\a_i$
Musta\cedilla{t}\v{a}-Nakamura's conjecture holds.
As the conjecture is equivalent to ACC conjecture in characteristic $0$,
the equality (\ref{equal}) gives ACC in positive characteristic. 
   
    \vskip.5truecm

\noindent
{\bf [Proof of Theorem \ref{lct}]} The statement is proved over the base field of characteristic $0$ 
       by Shibata 
       (see the proof of Theorem \ref{shibata} in the appendix).
       Let $k$ be an algebraically closed field of characteristic $p>0$ and 
       $\a\subset k[X_0,\ldots,X_N]$
       a nonzero proper ideal and assume $\frak m^\mu\subset \a\subset \frak m$,
       where $\frak m=(X_0,\ldots,X_N)$.
       Take a system of generators  ${\bf f}=\{f_1,\ldots, f_u\}$ of $\a$.
       We may assume that $\bf f$ contains  
       all monomials of degree $\mu$ in variables $X_1,\ldots, X_N$.
       Let $c:=\lct(0;\AA_k^N, \a)$.
       
       Now by the assumption that $(\AA_k^N,\a)\in \mathcal L_k$, 
       for every $\epsilon>0$ there exists a toric prime divisor $E_\epsilon$ over $\AA_k^N$
       with the center $0$ such that 
  \begin{equation}\label{lctepsilon1}
    \left|\frac{k_{E_\epsilon} +1}{\val_{E_\epsilon}\a}-c\right| <\epsilon.
\end{equation}
    Then, by Lemma \ref{keylemma}, we obtain a prime divisor $\we_{\epsilon}$ over $\AA_\CC^N$
    and a lifting  $\ta_\epsilon$ of $\a$
    such that 
    $$k_{E_\epsilon}=k_{\we_\epsilon}\ \ \mbox{and}\ \ 
    \val_{E_\epsilon}(\a)=\val_{\we_\epsilon}(\ta_\epsilon).$$
    On the other hand, 
    by the definition of $\ta_\epsilon$ in Lemma \ref{keylemma}
    we obtain 
\begin{equation}\label{contains}    
    \tm^\mu\subset \ta_\epsilon.
\end{equation}    

      Now by (\ref{lctepsilon1}), we obtain
 \begin{equation}\label{lctepsilon2}
    c+\epsilon>\frac{k_{E_\epsilon}+1}{\val_{E_\epsilon}\a}
    =\frac{k_{\we_\epsilon}+1}{\val_{\we_\epsilon}\ta_\epsilon }
    \geq \lct(0; \AA_\CC^N,\ta_\epsilon).
\end{equation}
     Here,  because of (\ref{contains}), we can apply Proposition \ref{shibata2} 
     to $(\AA_\CC^N, \ta_\epsilon)$ and by Lemma \ref{zm-ke} we obtain that the last term of the above inequalities is
     expressed as follows: 
     
\begin{equation}\label{lctepsilon3}
   \lct(0;\AA_\CC^N,\ta_\epsilon)=z_{m_\epsilon}(0;\AA_\CC^N,\ta_\epsilon)
\end{equation}
    for some $m_\epsilon\leq L'_{N,\mu}$ where $L'_{N,\mu}$ depends only on $N$ and $\mu$.
    By the same discussion as  about $s_m$  in Lemma \ref{localjet},
    we obtain
 \begin{equation}\label{lctepsilon4}
     z_{m_\epsilon}(0;\AA_\CC^N,\ta_\epsilon)\geq z_{m_\epsilon}(0;\AA_k^N,\a)\geq c.
 \end{equation}       
      By composing  (\ref{lctepsilon1}), (\ref{lctepsilon2}), (\ref{lctepsilon3}) and (\ref{lctepsilon4}) 
      we obtain 
\begin{equation}\label{lctepsilon5}
    c+\epsilon\geq z_{m_\epsilon}(0;\AA_k^N,\a)\geq c.
\end{equation}
     As $m_\epsilon$ is a positive integer bounded by $L'_{N,\mu}$,
     we have $m_\epsilon=m$ (constant) for sufficiently small $\epsilon$ and the equality
     $$z_m(0;\AA_k^N,\a)=c.$$
    This shows the existence of $m$ computing 
     $\lct(0;\AA_k^N,\a)$ 
    and also the uniform bound $L'_{N,\mu}$ of $m$.
    By Lemma \ref{zm-ke} this yields the bound $L_{N, \mu}$ of $k_E$ for some prime divisor $E$
    computing the log canonical threshold.
          $\Box$

 \vskip.5truecm

\noindent
{\bf [Proof of Corollary \ref{zhu}]} 
     Take a pair $(\AA_k^N, \a)\in \mathcal L_k$
     and assume that $\a $ is $\m$-primary.
     Then,  there exists $\mu\in \NN$ such that 
     $$\m^\mu\subset \a.$$
     Hence, by Theorem \ref{lct},
     we obtain a prime divisor over $\AA_k^N$ computing $\lct(0; \AA_k^N,\a)$.  
     $\Box$   

\vskip.5truecm
   Now we consider a generalization of the main theorems
   to the descent to the whole set of pairs in positive characteristic.
   First we pose a problem which asks if Lemma \ref{keylemma} holds for every prime divisor
   $E$ over $\AA_k^N$ with the center $0$.
\begin{prob}\label{problem}
   Let $k$ be an algebraically closed field of characteristic $p>0$,
   $(\AA_k^N, \a^e)$ a pair and $E$ a  prime divisor over $\AA_k^N$ with the center  $0$.
   Then, are there a prime divisor $\we$ over $\AA_\CC^N$ with the center 
    $0$ and multiideal $\ta^e$
    such that 
   $$k_E=k_\we, \ \ \val_E\a_i=\val_\we\ta_i,\ \ \  \mbox{and} \ \ 
   \ta_i (\mod p)=\a_i\  \mbox{ for\ all\ } i\ ?$$
\end{prob}   
Concerning with this problem, we obtain the following:
\begin{prop}
    If Problem \ref{problem} is affirmatively solved, 
    then Musta\cedilla{t}\v{a}-Nakamura's conjecture and ACC conjecture over $\CC$
     descend to
    the whole set of pairs $(\AA_k^N,\a^e)$ defined over an algebraically closed field
    $k$ of positive characteristic.
\end{prop}   

\begin{proof} About Musta\cedilla{t}\v{a}-Nakamura's conjecture,
   the proof is the same as the proof of Theorem \ref{mainthm} and
   we have only to replace the word ``toric prime divisor" by ``prime divisor".
   To show the descent of ACC conjecture, we assume that ACC conjecture holds
   over $\CC$.
     Let 
 \begin{equation}\label{ascending}     
      \mld(0;\AA_k^N, \a_{(1)}^{e_{(1)}})<\mld(0;\AA_k^N, \a_{(2)}^{e_{(2)}})<\cdots <
      \mld(0;\AA_k^N, \a_{(i)}^{e_{(i)}})< \cdots 
\end{equation}      
      be an ascending chain with $e_{(i)}\subset J$ and  multiideals  $\a_{(i)}$  of
      $k[X_1,\ldots,X_N]$
            for every $i\in \NN$,
      where $J$ is a DCC set.
     
      As we assume that  the conjecture ($A_N$) holds over the base field $\CC$, 
       by \cite[Theorem 4.6]{kawk2}, 
      the conjecture  $(M_{N,e})$ 
     holds over the base field $\CC$ for every finite set $e\subset \RR_{>0}$.
     Then, by the above discussion, for each $i\in \NN$ there exists a prime divisor  $E_i$ computing 
     $\mld(0;\AA_k^N, \a_{(i)}^{e_{(i)}})$.
     Therefore, if Problem \ref{problem} is affirmatively solved,
      there exist $\ta_{(i)}$ and $\we_i$ over
     the base field $\CC$  such that $\ta_{(i)} (\mod p)=\a_{(i)}$ and
     $$\mld(0;\AA_k^N, \a_{(i)}^{e_{(i)}})=a(E_i;\AA_k^N, \a_{(i)}^{e_{(i)}})
     =a(\we_i;\AA_{\CC}^N,\ta_{(i)}^{e_{(i)}})\geq \mld(0;\AA_{\CC}^N,\ta_{(i)}^{e_{(i)}}).$$
     By Lemma \ref{localjet}, we obtain  
     $$\mld(0;\AA_k^N, \a_{(i)}^{e_{(i)}})= \mld(0; \AA_{\CC}^N,\ta_{(i)}^{e_{(i)}}).$$
     Now,  the ascending chain (\ref{ascending}) coincides with the
     ascending chain on $\AA_\CC$:
     $$\mld(0;\AA_\CC^N, \ta_1^{e_{(1)}})<\mld(0;\AA_\CC^N, \ta_2^{e_{(2)}})<\cdots <
      \mld(0;\AA_\CC^N, \ta_i^{e_{(i)}})< \cdots .$$
      By the assumption $(A_{N})$ over $\CC$, this sequence stops at a finite stage,
      which yields $(A_N)$ over $k$. 
\end{proof}

\begin{prop} Assume that  Problem \ref{problem} is affirmatively solved and 
  moreover under the notation in Problem \ref{problem}
  assume that $\tilde\m^\mu\subset \ta$ when a pair $(\AA_k^N, \a)$
  satisfies $\m^\mu\subset \a$.
  Then,   for every pair $(\AA_k^N, \a)$ with an $\m$-primary ideal $\a$,
   there exists a prime divisor over $\AA_k^N$ computing  $\lct(0;\AA_k^N, \a)$.
\end{prop}

\begin{proof}
    We have only to replace the word ``toric prime divisor" in the proof of Theorem \ref{lct}
    by ``prime divisor"  and the observe that the proof of  Corollary \ref{zhu} works.
\end{proof}

The following is provided by Mircea Musta\cedilla{t}\v{a}.

\begin{prop}[M. Musta\cedilla{t}\v{a}]\label{mustata}
  Assume Problem \ref{problem} is affirmatively solved.
  Let $ \a$ be an ideal in $k[X_1,\ldots,X_N]$, where $k$ is an algebraically closed field of 
  characteristic $p>0$. Then either there is a divisor $E$ with center $0$ that computes 
  $c=\lct(0;\AA_k^N)$ or $c$ lies in the set $T_{N-1}$ of log canonical thresholds in characteristic 0, for ideals in $k[X_1,\ldots,X_{N-1}]$. 
  In any case, we get that the log canonical threshold of every ideal in positive characteristic is a rational number.
\end{prop}

\begin{proof}
In order to prove the  assertion, we use the fact that if $c$ does not lie in
$T_{N-1}$, then $c$ is not an accumulation point of log canonical thresholds in $T_N$.
Choose $\epsilon>0$ such that there is no element of $ T_N$ in $(c,c+\epsilon)$.
    Take a prime divisor $E$ over $\AA_k^N$ with the center  $0$ such that
$$\frac{k_E+1}{\val_E(\a)}<c+\epsilon,$$ 
 and then by choosing a prime divisor $\we$ over an ideal
$\ta$ in $\CC[X_1,\ldots,X_N]$ satisfying the conditions in the Problem \ref{problem}.
   Then 
   $$c+\epsilon>\frac{k_E+1}{\val_E(\a)}=\frac{k_\we+1}{\val_\we(\ta)}\geq \lct(0; \AA_\CC^N,\ta)\geq c.$$
   Here, the last inequality follows from 
   $$z_m(0,\AA_\CC^N,\ta)\geq z_m(0,\AA_k^N,\a)\ \ \mbox{for\ every}\ \ m\in \NN,$$
   which is proved in the same way as the inequality of $s_m$'s in Lemma \ref{localjet}.
  By the choice of $\epsilon$ we obtain
$$\lct(0;\AA_\CC,\a)=c. $$
This already shows that $c$ is a rational number.
Using again the fact that $c$ is not in $T_{N-1}$,
it follows that we have a prime divisor $F$ over $\AA_\CC^N$ 
with center $0$ that computes 
$\lct(0;\AA_\CC^N, \ta )$. 
In particular, there is $m\in \ZZ_{>0}$ such that $\lct(0;\AA_\CC^N,\ta)$ is computed by $z_m(0,\AA_\CC^N,\ta)$. Using the inequality
$z_m(0,\AA_\CC^N,\ta)\geq z_m(0,\AA_k^N,\a)$ 
and the fact that the two log canonical thresholds are equal, we see that in fact $\lct(0,\AA_k^N,\a)$ is computed by
$z_m(0,\AA_k^N,\a)$, and we obtain our assertion.  
\end{proof}

\begin{rem} Problem \ref{problem} may be solved by reducing the problem into the
   case of toric divisor by using the result that every discrete valuation can be lifted
   to a monomial valuation on an extended field (\cite{mono}).

\end{rem}

\newpage
\begin{center}
{\bf Appendix}
\vskip.5truecm
by Kohsuke Shibata, {\em University of Tokyo, Meguro, Tokyo, Japan}
\end{center}

\def\thesection{\Alph{section}}
\setcounter{section}{0}
\section{Proof of Proposition \ref{shibata2} in characteristic $0$}
\vskip.5truecm
\noindent
In this section, we prove that there is a bound of  discrepancies of prime divisors computing log canonical thresholds of ideals which contains a fixed power of the maximal ideal. 
We use the theory of generic limit of ideals to show this statement.
Throughout this section, we always assume that
the base field $k$ is an algebraically closed field of characteristic $0$.

We recall the definition of the generic limit of ideals  and some of its properties. 
The reader is referred to \cite{dFEM} for details.

Let $R=k[[X_1,\dots,X_N]]$ with the maximal ideal $\mathfrak m$ and for every field extension $L/k$ let
$R_L=L[[X_1,\dots,X_N]]$ and $\mathfrak m_L=\mathfrak mR_L$. 
For every $d$, let $\mathcal{H}_d$ be the Hilbert scheme parametrizing the ideals in $R$ containing $\mathfrak m^d$.
We have a natural surjective map $\tau_d: \mathcal{H}_d\to\mathcal{H}_{d-1}$
and
by generic flatness, 
we can
cover $\mathcal{H}_d$ by disjoint locally closed subsets such that the restriction of $\tau_d$ to each of these
subsets is a morphism.

We now fix  a positive integer $s$.
Consider the product $(\mathcal{H}_d)^s$ and the map $t_d: (\mathcal{H}_d)^s\to(\mathcal{H}_{d-1})^s$
that is given by $\tau_d$ on each component.

We define a generic limit of the collection  of $s$-tuples $\{({\mathfrak a_{1}^{(i)}},\dots,{\mathfrak a_{s}^{(i)}})\}_{i\in I}$ of  ideals in $R$ indexed by an infinite set $I$. we can construct irreducible closed subsets $Z_d\subset (\mathcal H_d)^s$
such that
\begin{itemize}%
\item[(1)]%
$t_d$ induces a dominant rational map $Z_d\dashrightarrow Z_{d-1}$,
\item[(2)]%
$I_d:=\{i\in I| (\mathfrak a_1^{(i)}+\mathfrak m^d,\dots,\mathfrak a_s^{(i)}+\mathfrak m^d)\in Z_d \}$ is infinite,
\item[(3)]%
the set of points in $(\mathcal H_d)^s$ indexed by $I_d$ is dense in $Z_d$.
\end{itemize}
Let $K:=\cup_{d\ge 1}k(Z_d)$. 
For each $d$, the morphism $\mathrm{Spec}K\to Z_d$ corresponds to an $s$-tuple $(\widetilde{\mathfrak a}_1^{(d)},\dots,\widetilde{\mathfrak a}_s^{(d)})$.
Furthermore there exist  ideals  ${\mathfrak a_j}$ in $R_K$ for $1\le j\le s$ such that $\widetilde{\mathfrak a}_j^{(d)}={\mathfrak a_j}+\mathfrak m_K^d$.
We call the $s$-tuple  $({\mathfrak a_1},\dots,{\mathfrak a_s})$ the generic limit of  the collection  of $s$-tuples $\{({\mathfrak a_{1}^{(i)}},\dots,{\mathfrak a_{s}^{(i)}})\}_{i\in I}$.

\begin{lem}\label{inclusion of generic limit}
{\rm[Lemma 3.1 and Lemma 3.2  in \cite{dFEM}]}
      With the above notation, suppose that
      $\mathfrak b\subset\mathfrak a_j^{(i)}\subset \mathfrak m$ for some ideal $\mathfrak b\subset R$
      and every $i,j$.
     Then $\mathfrak bR_K\subset\mathfrak a_j\subset \mathfrak m_K$ for every $j$.
\end{lem}

     We put $X=\mathrm{Spec}\, R$ and 
     $\widetilde{X}= \mathrm{Spec}R_K$.
     We denote by $x$ and $\tilde{x}$ 
     the closed points of $X$ and   $\widetilde{X}$, respectively.

\begin{prop}\label{generic limit}
{\rm[Proposition 3.3  in \cite{dFEM}, Proposition 3.1  in \cite{mn}]}
With the above notation, suppose that $\mathfrak a_j^{(i)}$ are proper nonzero ideals.
We write  ${(\mathfrak a^{(i)})}^e:=\prod_{j=1}^s(\mathfrak a_j^{(i)})^{e_j}$ and ${\mathfrak a}^e:=\prod_{j=1}^s\mathfrak a_j^{e_j}$ for  some positive real numbers $e=\{e_1,\dots,e_s\}$.
Suppose that $\mathfrak a_j\neq 0$ for all $j$.
Then the following hold:
\begin{itemize}%
\item[(i)]%
For every positive real numbers $e=\{e_1,\dots,e_s\}$,
then $\mathrm{lct}({\widetilde{x}};\widetilde X, {\mathfrak a}^e)$ is a limit point of the set 
$\{\mathrm{lct}(x;X,{(\mathfrak a^{(i)})}^e)|\, i\ge 1, i\in I\}$.
 
\item[(ii)]%
      If $E$ is a prime divisor over $\widetilde{X}$  computing $\lct({\widetilde{x}};\widetilde X,    
      {\mathfrak a}^e)$ for some positive real numbers $e=\{e_1,\dots,e_s\}$ and having center at the closed    point $\widetilde{x}$, then for every 
      $d\gg 0$ there is an infinite subset $I_d'\subset I$ depending on $E$ and $e$ and 
      satisfying  the following property:
     for every $i\in I'_d$ there is a divisor $E_i$ over $X$ that computes 
     $\mathrm{lct}(x;X, \prod_{j=1}^s(\mathfrak a_j^{(i)}+\mathfrak m^d)^{e_j})$, 
     which is equal to  
     $\mathrm{lct}({\widetilde{x}};\widetilde X, \prod_{j=1}^s(\mathfrak a_j+\mathfrak m_K^d)^{e_j})$ and we have $\val_E({\mathfrak m}_K)=\val_{E_i}({\mathfrak m})$, $k_E=k_{E_i}$ and  $\val_E({\mathfrak a_j}+{\mathfrak m}_K^d)=\val_{E_i}({\mathfrak a_j^{(i)}}+{\mathfrak m}^d)$ for $1\le j\le s$.

\end{itemize}
\end{prop}

\begin{defn}
   Let $X$ be a log terminal variety,
   $x$ a closed point of $X$  
    and $\mathfrak b$ a proper nonzero ideal on $X$.
    The minimal discrepancy of $X$ computing $\mathrm{lct}(x; X,\mathfrak b)$ is defined by
$$\mathrm{md}(\mathrm{lct}(x;X,\mathfrak b)):=\mathrm{min}\{k_E \mid E\ \mathrm{computes}\ \mathrm{lct}(x;X,\mathfrak b)\}.$$
\end{defn}

\begin{thm}\label{shibata}
 Let $R=k[[X_1,\dots,X_N]]$ with the maximal ideal $\mathfrak m$.
Let $X=\mathrm{Spec}\, R$ and $x$ the  closed point of $X$.

    Then, for  a positive integer $\mu$,
    there is a positive integer $L_{N,\mu}$ depending only on $N$ and $\mu$ such that for every 
    ${\mathfrak m}$-primary ideal $\mathfrak b$ on $X$ with 
       ${\mathfrak m}^\mu\subset \mathfrak b$, 
    there is some prime divisor $E$ over $X$ that computes $\mathrm{lct}(x;X,\mathfrak b)$  
    and satisfies  $k_E\le L_{N, \mu}$.
\end{thm}

\begin{proof}
    We argue by contradiction. 
    If the conclusion of the theorem fails, then
    we can find a sequence of ${\mathfrak m}$-primary ideals  $({\mathfrak a}_i)_{i\ge 1}$ such that   
    for every $i$, 
    $\m^\mu\subset\mathfrak a_i\subset \m$ and 
    $$\mathrm{lim}_{i\to\infty}\mathrm{md}(\mathrm{lct}(x;X, \mathfrak a_i))=\infty
    .$$
      Let $\mathfrak a\subset K[[X_1\dots,X_n]]$ be the generic limit of $(\mathfrak a_i)_{i\ge 1}$.
      Let $R_K=K[[X_1,\dots,X_N]]$ with the maximal ideal $\mathfrak m_K$.
     We put $\widetilde{X}= \mathrm{Spec}R_K$.
    Let  $\widetilde{x}$ be the closed point 
    of  $\widetilde{X}$.
    Then by  Lemma \ref{inclusion of generic limit}, 
    we have  $${\mathfrak m}_K^{\mu} \subset{\mathfrak a}\subset\mathfrak m_K.$$ 
     Since ${\mathfrak a}$ is an ${\mathfrak m}_K$-primary ideal, 
     any divisor over  $\widetilde{X}$ computing  $\mathrm{lct}(\widetilde{x};\widetilde X,
     {\mathfrak a})$  has the center at the closed point $\widetilde{x}$.
      Let $E$ be a  divisor over $\widetilde{X}$  computing 
      $\mathrm{lct}({\widetilde{x}};\widetilde X,{\mathfrak a})$.
   
   Then by  Proposition \ref{generic limit}, 
    for every $d\gg \mu$ there is an infinite subset $I'_d\subset \mathbb N$ 
    with the following property:
    for every $i\in I'_d$ there is a divisor $E_i$ over $X$ that computes 
    $\mathrm{lct}(x;X,{\mathfrak a}_i)$, which is equal to  $\mathrm{lct}({\widetilde{x}};
    \widetilde X, {\mathfrak a})$ and we have $$k_E=k_{E_i}\ \mbox{for\ every\ } i\in I'_d.$$
    This is a contradiction to the assumption that  $\mathrm{lim}_{i\to\infty}\mathrm{md}(\mathrm{lct}(x;X, \mathfrak a_i))=\infty$.
\end{proof}

  The above theorem does not hold if we do not assume  the inclusion $\m^\mu\subset \a$
    of ideals for the fixed power $\mu$.

\begin{exmp}
    Let $\mathfrak a_i=(x,y^i)\subset k[x,y]$ for $i\in\mathbb N$.
    Then by the main theorem in \cite{H}, we have
    $$\mathrm{lct}(0; \AA_k^2, \mathfrak a_i)
    =\mathrm{sup}\{t\in \mathbb R_{\ge 0}\mid  (1,1)\in t\cdot\mathrm{Newt}(x,y^i)
   \},$$
      where $\mathrm{Newt}(x,y^i) \subset\mathbb R_{\ge 0}^2$ is 
       the Newton polygon spanned by $x$ and $y^i$.
      Therefore $$\mathrm{lct}(0; \AA_k^2, \mathfrak a_i)=(i+1)/i.$$
      Let $E_i$ be a prime divisor over $\AA_k^2$ computing  
     $\mathrm{lct}(0; \AA_k^2, \mathfrak a_i)$.
      Then $$({k_{E_i}+1})/{\mathrm{val}_{E_i}(\mathfrak a_i)}=(i+1)/i.$$
      We can rewrite this as $i({k_{E_i}+1})=(i+1){\mathrm{val}_{E_i}(\mathfrak a_i)}.$
      Note that $k_{E_i}$ and $\mathrm{val}_{E_i}(\mathfrak a_i)$ are integers.
      Since $i$ and $i+1$ are coprime, the integer $k_{E_i}+1$ is divisible by $i+1$.
      Therefore  $k_{E_i}+1\ge i+1$.
      This implies that $\mathrm{lim}_{i\to\infty}\mathrm{md}(\mathrm{lct}(0; \AA_k^2,\mathfrak a_i))=\infty$.
\end{exmp}

\begin{rem} 
      In the theorem we can replace the condition $\m^\mu\subset \b$ by the following:
      $$\ell(\o_X/\b)\leq \mu,$$
      because from this inequality, the condition $\m^\mu\subset \b$ follows.
      
      By the same reason, it is also possible to replace the condition  
      $\m^\mu\subset \b$ by the following inequality of the multiplicity $e(\b)$ of $\b$: 
      $$e(\b)\leq \mu.$$

\end{rem}


\vskip1truecm

\noindent Shihoko Ishii, \\ Department of Mathematics, Tokyo Woman's Christian University,\\
2-6-1 Zenpukuji, Suginami, 167-8585 Tokyo, Japan.\\

\end{document}